\documentclass[reqno]{amsart}
\numberwithin{equation}{section}

\title{Maximal estimates for Weyl sums on $\mathbb{T}^{d}$}
\date{}

\usepackage{mathtools}
\usepackage{indentfirst}
\usepackage{amsmath}
\usepackage{dsfont}
\usepackage{amssymb}
\usepackage{amsthm}
\usepackage{hyperref}
\usepackage{color}

\usepackage{bbold}

\usepackage{cases}

\newcommand{\R}{\mathbb{R}}
\newcommand{\Z}{\mathbb{Z}}
\newcommand{\N}{\mathbb{N}}

\newcommand{\T}{\mathbb{T}}
\newcommand{\eps}{\varepsilon}

\begin{document}

\address{Department of Mathematics, University of Illinois Urbana-Champaign, Urbana, IL 61801, USA}
\email{aabarron@illinois.edu}

\author{Changxing Miao}
\address{Institute of Applied Physics and Computational Mathematics, Beijing 100088, China}
\email{miao\_changxing@iapcm.ac.cn}

\author{Jiye Yuan}
\address{Department of Mathematics, Beijing Key Laboratory on Mathematical Characterization, Analysis, and Applications of Complex Information, Beijing Institute of Technology, Beijing 100081, China}
\email{yuan\_jiye@bit.edu.cn}

\author{Tengfei Zhao}
\address{School of Mathematics and Physics, University of Science and Technology Beijing, Beijing 100083, China}
\email{zhao\_tengfei@ustb.edu.cn}

\newtheorem{thm}{Theorem}[section]
\newtheorem{lemma}[thm]{Lemma}
\newtheorem{cor}[thm]{Corollary}
\newtheorem{prop}[thm]{Proposition}
\newtheorem{remark}[thm]{Remark}
\newtheorem{defi}[thm]{Definition}
\newtheorem{conj}[thm]{Conjecture}
\renewcommand\theequation{\thesection.\arabic{equation}}

\newenvironment{equ}{\begin{equation}}{\end{equation}}
\newenvironment{equ*}{\begin{equation*}}{\end{equation*}}

\newtheorem{innercustomprop}{Proposition}
\newenvironment{customprop}[1]
  {\renewcommand\theinnercustomprop{#1}\innercustomprop}
  {\endinnercustomprop}

 \newtheorem{innercustomlemma}{Lemma}
\newenvironment{customlemma}[1]
  {\renewcommand\theinnercustomlemma{#1}\innercustomlemma}
  {\endinnercustomlemma}

\maketitle

\vspace{-.5cm}
{\centerline{(With an Appendix by Alex Barron)}}

\begin{abstract}
In this paper, we obtain the maximal estimate for the Weyl sums on the torus $\T^d$ with $d\geq 2$, which is sharp up to the endpoint.
We also consider two variants of this problem which include the maximal
estimate along the rational lines and on the generic torus.
Applications, which include some new upper bound on the Hausdorff dimension of the sets
associated to the large value of the Weyl sums, reflect the compound phenomenon between the square root cancellation and the constructive interference.
In the Appendix, an alternate proof of Theorem \ref{T-d} inspired by Baker's argument in \cite{Baker}
is given by Barron, which also improves the $N^{\epsilon}$ loss in Theorem \ref{T-d},
and the Strichartz-type estimates for the Weyl sums with logarithmic losses are obtained by the same
argument.
\end{abstract}

\vspace{0.5cm}

\begin{center}
 \begin{minipage}{100mm}
   { \small {{\bf Key Words:}  Periodic Schr\"odinger equation; Weyl sums; Maximal estimate; Hausdorff
   dimension; Strichartz-type estimate.
   }
      {}
   }\\
    { \small {\bf AMS Classification:}
      {42B25,  42B37, 35Q41.}
      }
 \end{minipage}
 \end{center}

\section{Introduction}\label{intro}
In this paper, we study the maximal estimate
\begin{equation}
  \Big\|\sup_{0<t<1}|u(x,t)|\Big\|_{L^{p}(\T^d)} \lesssim N^{s}\Big(\sum_{n\in\Z^d}|a_n|^{2}\Big)^{\frac{1}{2}}
\end{equation}
of the following function
\begin{equation}\label{torus-sln}
u(x,t)=\sum_{n\in\Z^d}a_{n}e^{2\pi i(n\cdot x+|n|^{2}t)}
\end{equation}
for $1\le p<\infty$ and $s>0$,
which is the solution of the Schr\"odinger equation
\begin{equation}
    i\partial_{t}u - \Delta u=0
\end{equation}
with the periodic boundary condition
\begin{equation}
  u(x,0) = \sum_{n\in \mathbb Z^d}a_{n}e^{2\pi in\cdot x}.
\end{equation}


One of the prominent applications of the maximal estimates is the pointwise convergence
of the solution to the corresponding equations.
In the Euclidean setting, in 1980 Carleson in \cite{Car} first put out a question about
finding the minimal $s$
such that for $\forall f\in H^{s}(\R)$ there holds
\begin{equation}\label{pointwise-E}
  \lim_{t\rightarrow0}e^{it\Delta_{\R}}f(x) = f(x), \quad\text{for\,\,a.e.}\,\,x\in\R,
\end{equation}
and he  proved \eqref{pointwise-E} holds for $s\ge\frac{1}{4}$ in $\R$.
Later Dahlberg and Kenig in \cite{D-K} proved that the result obtained by Carleson is sharp.

For the higher dimension cases, the geometric structure is quite different and seems more complicated.
In dimensions $d\ge2$, Sj\"olin in \cite{Sjo} and Vega in \cite{Vega}
established the pointwise convergence for $s>\frac{1}{2}$ independently.
With the progress of Fourier restriction theory for the paraboloid or the sphere,
many harmonic analysis tools are also applied to the pointwise convergence problem.
By employing the localization lemmas and a crucial relation (defined in \cite{Tao})
between tubes and cubes corresponding to the wave packet decomposition, Lee in \cite{Lee} proved the pointwise convergence holds for
$s>\frac{3}{8}$ in $\mathbb{R}^2$.
On the other hand, for $d\ge 2$, by making use of the multilinear estimates for the Fourier extension operator,
Bourgain in \cite{Bourgain-0} proved $s>\frac{1}{2}-\frac{1}{4d}$.
For negative results, Bourgain \cite{Bourgain-2} proved the necessity of $s\ge\frac{d}{2(d+1)}$ with $d\ge2$ for the pointwise convergence and Luc\`a and Rogers in \cite{LR} gave an alternative proof.

Recently, by establishing the refined Strichartz estimates,
Du-Guth-Li in \cite{DGL} proved the result of $s>\frac{1}{3}$ in dimension $d=2$  via polynomial partition  and $\ell^2$ decoupling estimates, see \cite{BD}.
Later, Du-Zhang in \cite{DZ} obtained that pointwise convergence holds for $s>\frac{d}{2(d+1)}$ with $d\ge3$, which is sharp up to the endpoint.
The refined Strichartz estimates are closely related to the
the refined decoupling estimates, see Du-Guth-Li-Zhang \cite{DGLZ},
Demeter \cite{Demeter} and Guth-Iosevich-Ou-Wang \cite{GIOW}.


%

As in $\mathbb R^d$, considering the counterexample in Subsection \ref{lower-bound-for-certain-functions} as well as in \cite{DKWZ-2020} and \cite{DZ}, one may conjecture that the maximal estimate in $\T^d$
\begin{equation}\label{m-3}
\Big\|\sup_{0<t<1}\big|\sum_{|n|\le N} a_{n}e^{2\pi i(n\cdot x+n^{2}t)}\big|\Big\|_{L^{p}(\T^d)}
\lesssim N^{s_p}(\sum_{|n|\le N }|a_{n}|^{2})^{\frac{1}{2}},
\end{equation}
holds for $1\le p \le 2$ and $s_p>\frac{d}{2(d+1)}$.

From preliminary analysis, compared to the Euclidean setting,  one may find that the resonances in the periodic setting are more intense, so that maximal estimates \eqref{m-3} seem harder to justify.
More precisely, the time localization arguments play an important role in the proof for the maximal estimate of the Schr\"odinger equaton in $\R^d$, however, is invalid in $\T^d$, see \cite{Bourgain-0, DGL, DZ, Lee, M-V}.
From the Strichartz estimates in $\mathbb T^d$,
Moyua and Vega in \cite{M-V},
Compaan, Luc\`{a}, and  Staffilani
 in \cite{CLS-2021}  proved that \eqref{m-3} holds for
  $1\le p\le \frac{d+2}{d}$ and $s_p>\frac{d}{d+2}$, which
seems the best result
for the general initial value of the periodic Schr\"odinger equation
 up to now.



On the other hand, Barron \cite{A-B}
studied the maximal estimate \eqref{m-3} for a special case ($a_{n}=1$) on $\T$
via induction on scales
based on the classical Weyl estimates and the Hardy-Littlewood circle method.
The key ingredients are exploring and  handling the large values of the Weyl sums by
better Diophantine approximation properties.
%
%
Later, Baker in \cite{Baker} gave another proof based on the
Dirichlet approximation lemma for time $t$ and
refined upper bound of  Weyl sums,
which  
exposes some direct information of the spatial variable $x$.  
Moreover,  Baker, Chen and Shparlinski \cite{BCS1, BCS2} consider generalized maximal estimates of Weyl sums of the degree $d$.




%

In this paper, we first consider the maximal estimate for the Weyl sums in higher dimensions for $d\ge2$.

\begin{thm}\label{T-d}
\begin{equation}\label{max-Lp}
\bigg\|\sup_{0<t<1}\Big| \prod_{1\leq j\leq d} \sum_{1\le n_{j}\le N}e^{2\pi i(n_jx_j+|n_j|^2t)}\Big|\bigg\|_{L^{p}(\mathbb{T}^{d})}
\lesssim N^{\frac{d}{2}+s_p},
\end{equation}
where $s_p>\frac{d}{2(d+1)}$ when $1\leq p \leq \frac{2(d+1)}{d}$ and
$s_p>\frac{d}{2}-\frac{d}p$ when $ p \geq \frac{2(d+1)}{d}$.

\end{thm}

We prove  this theorem by analysing  the set where
large values of the Weyl sums are obtained, inspired by the methods of \cite{Baker,A-B}.
More precisely, by the Dirichlet theorem and the classical Weyl estimates, we find that the Weyl sums become large for certain time $t$ which is in a neighborhood of $a/q$ with coprime $q$ satisfying $1\leq a<q$.
On the one hand, if there are some separating Weyl sums with respect to one direction $x_j$ of $x$ which has the upper bound $N^{\frac{1}{2}+\frac{1}{2d(d+1)}}$, then applying induction on the dimension $d$, we would prove the maximal estimate \eqref{max-Lp}.
On the other hand, if all separating Weyl sums in every direction have the lower bound
$N^{\frac{1}{2}+\frac{1}{2d(d+1)}}$,
then from a Weyl-type lemma in \cite{Bourgain-1} and
the refined Diophantine approximation in \cite{Baker,A-B},
one can utilize the Hardy-Littlewood circle method to deduce some additional restrictions on the parameter $q$.
Then, by utilizing the refined Weyl sums estimates of \cite{Schmidt}, we can
obtain the restricted region of the spatial variable $x$ in each dimension,
and conclude a refined upper bound on the  Weyl sums with respect to $x$.

In the Appendix, based on the Baker's argument in \cite{Baker},
another proof of Theorem \ref{T-d} is given by Barron, which improves the $N^{\epsilon}$ loss.
The similar argument also implies the Strichartz-type estimates for the Weyl sums.



\begin{remark}
  In view of  a simple counterexample, we see that for $p\ge\frac{2(d+1)}{d}$, the maximal estimate \eqref{max-Lp} is almost sharp up to the endpoint.
Besides, for $1\le p\le\frac{2(d+1)}{d}$,  the loss in $N$ for \eqref{max-Lp} cannot be improved even if $p$ is decreased.
In fact,  one can construct a set $E\subset\T^d$ with $|E|\gtrsim1$ such that for $x\in E$,
$$\sup_{0<t<1} \bigg| \prod_{1\leq j \leq d} \sum_{n_j=1}^{N} e^{2\pi i(n_{j} x_{j}+|n_{j}|^2t)} \bigg| \gtrsim N^{\frac{d}{2}+\frac{d}{2(d+1)}}.$$ See Section 4 for detail.


\end{remark}

Besides the maximal estimates \eqref{max-Lp}, we consider two invariant versions of the maximal estimates for the Weyl sums related to the pointwise convergence along the line $(x-rt,t)$ for $r\in\mathbb{Q}^d$ with $0<t<1$ and the maximal
estimates on the irrational tori defined below.

We first introduce the following maximal estimates.

\begin{thm}\label{T-d-2}
Let $r=(r_1,r_2,\cdots,r_d)\in\mathbb{Q}^{d}$, then we have
  \begin{equation}\label{max-ra}
  \Big\|\sup_{0<t<1} \big| \prod_{1\leq j\leq d}\sum_{1\le n_{j}\le N}e^{2\pi i(n_j(x_j-r_{j}t)+|n_j|^2t)}\big| \Big\|_{L^{p}(\mathbb{T}^{d})}
  \lesssim N^{\frac{d}{2}+s_p},
  \end{equation}
  where $s_p>\frac{d}{2(d+1)}$ when $1\leq p \leq \frac{2(d+1)}{d}$ and
  $s_p>\frac{d}{2}-\frac{d}p$ when $ p \geq \frac{2(d+1)}{d}$.
\end{thm}


\begin{remark}
(i)
   As a consequence of the maximal estimates \eqref{max-ra}, the Weyl sums converge to the initial datum along the lines
      $(x-rt,t)$ for $r\in\mathbb{Q}^d$ pointwise.
  In particular, the maximal estimates \eqref{max-Lp} is the special case of \eqref{max-ra} for $r=0$ which shows the pointwise convergence along the vertical line.

(ii)
  In $\R$, using stationary phase analysis and the $TT^{\ast}$ argument, Shiraki in \cite{SS} and together with Cho in \cite{Ch-Sh} obtained the pointwise
             convergence for the Schr\"odinger equation along the non-tangential lines and the tangential lines respectively.
  The authors in \cite{YZ} and
             together with Zheng in \cite{YZZ1,YZZ2} studied these several variants of
             the pointwise convergence problem for the fractional
             Schr\"odinger operator with complex time.
On $\T$, Baker, Chen and Shparlinski in \cite{BCS1} dealt with the maximal estimates for the Weyl sums
along rational and irrational non-tangential lines.
However, because of the different properties
       between the space $\T^d$ and $\R^d$ such as Dirichlet lemma and Hardy-Littlwood circle method, here we only  obtain the pointwise
             convergence along rational non-tangential lines on $\T^d$ for $d\ge2$.
One may also consider the pointwise convergence of the solutions to the Schr\"odinger equation on $\mathbb T^d$
along general curves via analysing the maximal estimates for Weyl sums.


\end{remark}

Next, we  consider
the maximal estimates for the Weyl sums on the generic torus. 
Without loss of generality, we assume that the
each component parameter $\beta_{i}$ of $\beta$ satisfies
$$
\beta_{i}\in[1,2], \quad\forall i\in\{1,\cdots,d\}.
$$
We recall the definition of genericity which one can find in \cite{DGG}.
\begin{defi}\label{generi}
  We will call a property generic in $(\beta_1,\cdots,\beta_d)$ if it is true for all
  $(\beta_1,\cdots,\beta_d)$ outside of a null set (set with measure zero) of $[1,2]^{d}$.
\end{defi}

Here, the definition of genericity  of $\beta$ is based on the  classical result on Diophantine approximation. We call $\beta=(\beta_2,\cdots,\beta_{d})$ is generic if there exists some constant $C$ such that
\begin{equation}\label{genericity}
|k_1+\beta_2 k_2+\cdots+\beta_d k_d | \geq \frac{C}{(|k_1|+\cdots |k_d|)^{d-1} \log (|k_1|+\cdots |k_d|)^{2d} }
\end{equation}
for any $k=(k_1,\cdots,k_d)\in \mathbb Z^d$,
see \cite{Cas}. 

We obtain the maximal estimates of Weyl sums on generic tori as follows.

\begin{thm}\label{T-d-3}
  For generic $\beta_2,\cdots,\beta_d$ in $[1,2]$ 
  and polynomial $Q(n) = |n_1|^{2} + \beta_2|n_2|^{2} + \cdots +\beta_d|n_d|^2$, we have
  \begin{equation}\label{max-irra}
    \Big\|\sup_{0<t<1}\big|\sum_{n\in[-N,N]^{d}\cap\mathbb{Z}^d}e^{2\pi i(x\cdot n+tQ(n))}\big|\Big\|_{L^{2}(\T^d)} \lesssim N^{\frac{d+1}{2}+\epsilon}.
  \end{equation}
\end{thm}


\begin{remark}
  The maximal estimate \eqref{max-irra} extends the maximal estimates for the Laplacian operator $\Delta_{\T^d}$ to the generic
case $Q(D)$ associated with the following Schr\"odinger equation,
\begin{equation}\label{Q-n}
  \begin{cases}
    i\partial_t u + Q(D)u=0,\\
    u(x,0) = \sum\limits_{n\in[-N,N]^{d}\cap\mathbb{Z}^d}e^{2\pi ix\cdot n}.
  \end{cases}
\end{equation}
The sharp maximal estimate for the solution to the
 equation \eqref{Q-n} with the general initial value also remains open.
\end{remark}

The main ingredients of proof for Theorem \ref{T-d-3} are
the maximal estimate for $\frac{2}{N}<t<1$ which is related to the Strichartz estimates in
\cite{DGG} and local-in-time maximal estimate \eqref{local-tori-0} for $0<t<\frac{2}{N}$
in Section 3, see Lemma \ref{local-max}. 





\subsection{Applications: large values of the Weyl sums}

Let
\begin{equation}\label{omega-N}
\begin{split}
\omega_{N}(x,t)
&:= \sum_{\substack{1\le n_{j}\le N,\\1\le j\le d}}e^{2\pi i(n\cdot x+|n|^{2}t)}\\
&= \prod_{j=1}^{d}(\sum_{1\le n_j\le N}e^{2\pi i(n_{j}x_{j}+n_{j}^{2}t)})\\
&:=\prod_{j=1}^{d}\omega_{N,j}(x_j,t),
\end{split}
\end{equation}
where $n=(n_1,n_2,\cdots,n_d)$.

By the maximal estimates \eqref{max-Lp}, we give some applications associated with the large values of the Weyl sums. 
One application is to estimate the Lebesgue measure of the set where Weyl sums attain large value and
are bounded from below by $N^{\alpha}$ in Section 5 as follows,
\begin{equation}
  S_{\alpha}(N) := \{x\in\T^d:\sup_{0<t<1}|\omega_{N}(x,t)|\ge N^{\alpha} \},
\end{equation}
where $0<\alpha\le d$.
%
%
We extend the results of Chen and Shparlinski in \cite{CS} and Barron in \cite{A-B} in $\T$
to the higher dimension $\T^d$ for $d\ge2$, and obtain
\begin{equation}\label{S-alpha-L}
  |S_{\alpha}(N)|\lesssim_{\epsilon} N^{(d+2)-\frac{2(d+1)\alpha}{d}+\epsilon}
\end{equation}
for $\frac{d}{2}+\frac{d}{2(d+1)}<\alpha\le d$.

The estimate \eqref{S-alpha-L} for $S_{\alpha}(N)$ characterizes the
extent of the occurrence for
the square root cancellation for the Weyl sums $\omega_{N}(x,t)$. It is easy to see that
when $(x,t)$ equals $(0,0)\in\T^{d+1}$ or approximates $(0,0)$, constructive interference
for $\omega_{N}(x,t)$ occurs, that is
\begin{equation}
  |\omega_{N}(x,t)| \sim N^{d} \quad \text{for}\,\,(x,t)\in\mathcal{N}_{N^{-2}}(0,0).
\end{equation}
In the pointwise sense, the square root cancellation which corresponds to $|\omega_{N}(x,t)|\sim N^{\frac{d}{2}}$
hardly happens.
However, our result about the estimate for $S_{\alpha}(N)$ falling in
 between these two cases can be obtained, which exhibits the compound effect between the constructive
 interference and the square root cancellation for Weyl sums.

In Section 6, as the second application corresponding to the large values of the Weyl sums,
we consider the Hausdorff dimension of the set on which the Weyl sums are large and have the
lower bound $N^{\alpha}$ for many $N$, that is
\begin{equation}\label{L-alpha}
  L_{\alpha}:=\{x\in\T^d:\sup_{0<t<1}|\omega_{N}(x,t)|\ge N^{\alpha}\,\,\text{for infinitely many}\,\,N\in\N\}.
\end{equation}

In one dimension, utilizing the completion method, Chen and Shparlinski in \cite{CS2, CS3} and
 Barron in \cite{A-B} studied the corresponding set.
 Using the higher dimension version of the completion method and the fractal version of the maximal estimates \eqref{max-Lp},
 we have the following estimate for the set $L_{\alpha}$.
%
\begin{prop}\label{Hau-1}
If $\frac{d}{2}+\frac{d}{2(d+1)}\le\alpha\le d$, then
\begin{equation}
  \dim_H(L_{\alpha}) = \tfrac{2(d+1)}{d}(d-\alpha).
\end{equation}
\end{prop}

\begin{remark}
For $\alpha=d$,
$\dim_{H}(L_{\alpha}) = 0$ indicates that the set where the
constructive interference for the Weyl sums occurs is a small part in $\T^d$.
While $\dim_{H}(L_{\alpha}) = d$ for $\alpha = \frac{d}{2}+\frac{d}{2(d+1)}$ shows
that at almost every point, the square root cancellation happens.
\end{remark}

\subsection{Notation}

We write $A\lesssim B$ to mean that there exists a constant $C>0$ such that $A\le CB$.
If there exists a constant $C(\beta)$ that depends on some parameter $\beta$ such that $A\le C(\beta)B$, then we write $A\lesssim_{\beta} B$.
We write $A\sim B$ if $A\lesssim B$ and $A\gtrsim B$.
 We use $A\ll B$ to denote the statement $A\le cB$ where $c$ is a very small constant.
 Throughout the paper, we often omit the constant $C$,  $\beta$ and $c$ which may differ from line to line according to the context.

Let $E$ be a Borel set in $\R^d$. For $0\le s\le d$, the $s$-dimensional Hausdorff measure  of $E$ is defined as follows
\begin{equation}
  \mathcal{H}^{s}(E) :=
  \lim_{\delta\rightarrow0}\mathcal{H}^{s}_{\delta}(E),
\end{equation}
where for $0<\delta\le\infty$,
\begin{equation}
  \mathcal{H}^{s}_{\delta}(E) := \inf\{\sum_{j}d(B_{j})^{s}: E\subset\cup_{j}B_{j}, d(B_{j})<\delta\}.
\end{equation}
We define the Hausdorff dimension of the Borel set $E$ by
\begin{equation}
  \dim_{H}(E) := \inf\{0\le s\le d: \mathcal{H}^{s}(E) = 0 \} = \sup\{0\le s\le d: \mathcal{H}^{s}(E) = \infty\}.
\end{equation}

In Section 2 we give the proof of Theorem \ref{T-d} and Theorem \ref{T-d-2}. In Section 3 we give the proof of Theorem \ref{T-d-3}.

\section{Proof of Theorem \ref{T-d} and Theorem \ref{T-d-2}} 

In this section, we prove Theorem \ref{T-d} and Theorem \ref{T-d-2}.
Enlightened by the method developed by Baker in \cite{Baker} who dealt with the case $d=1$, we give the proof in higher dimension.
Nonetheless, the method in \cite{Baker} cannot  be applied to deal with the case in higher dimension directly, and we need to investigate a better upper bound for the Weyl sums to give a better control.
Also we can see that the method for the proof of Theorem \ref{T-d} and Theorem \ref{T-d-2} in higher dimension for $d\ge2$ in this section can be applied to the case for $d=1$.

Before the proof, we supply some preliminaries.
We write $u(x,t) = \omega_{N}(x,t)$ as in Section 1.1 for the following sections.
\subsection{Preliminaries}\label{preli}

Let $k\geq 2$ and $K=2^{k-1}$.  Let $\|\alpha \|$ be the distance between  $\alpha \in \mathbb{R}$ and $\mathbb{Z}$.
Denote
$$ S(f) :=\sum_{1\leq n \leq N}e(f(n))  =\sum_{1\leq n \leq N}e^{2\pi if(n)}.$$

\begin{lemma}[Bourgain \cite{Bourgain-1}, Lemma 2.2 of Barron \cite{A-B}]\label{w-1}
Let $q,a$ be coprime, with $1\leq a< q\leq N$. Suppose that $|t-\frac{a}{q}|<\frac{1}{qN}$, then we have
\begin{equation}\label{w-1-B}
|S(x n+tn^2)| \ll N^{1+\eps} \min(q^{-\frac12}, N^{-1}|qt-a|^{-\frac12}).
\end{equation}
The constant implied by $\ll$ depends only on $\eps$ with $0<\eps<\frac{1}{1000}$.
\end{lemma}

By this lemma, we can obtain a better result for the approximation property for time $t$
than that Dirichlet lemma shows, if the Weyl sums have a nontrivial lower bound.
Indeed, utilizing this lemma, we have if  $P\geq N^{\frac12+\delta}$ and
\begin{equation}\label{P-1}
P\leq |S(x n+tn^2)| ,
\end{equation}
for some $\delta>0$, then
\begin{equation}\label{q-1}
1\leq q \leq (N P^{-1})^2 N^{2\eps}\leq N^{1-2(\delta-\eps)} \text{  and } |qt-a|\leq N^{2\eps} P^{-2} \leq N^{-1-2(\delta-\eps)},
\end{equation}
where $0\ll \eps \ll \delta< 1$.

Furthermore, the following lemmas would supply  additional constraints on the other variable $x$, which satisfies the large condition \eqref{P-1}.
\begin{lemma}[Lemma 10D in  Schmidt \cite{Schmidt}]\label{w-2}
Suppose $\eps > 0, k \geq 2$ and $L < N$.  Recall $K=2^{k-1}$.
Suppose $f(x) = \alpha x^k + \beta x^{k-1} + \gamma x^{k-2} +\cdots+c_0.$
Then if we write $S_m(f) = S(mf)$ for integers $m$, we have
\begin{equation}\label{}
\sum_{m=1}^L |S_m(f)|^K \ll N^{K-k+\eps} \sum_{z=1}^{k!N^{k-2}L}
\,\,\sum_{u=1}^{2kN} \min(N,\|z(u \alpha+2\beta)\|^{-1}).
\end{equation}
The constant implied by $\ll$ depends only on $k, \eps$.
\end{lemma}

With  Lemma \ref{w-2} in mind,  if we take $k=K=2$ and $L=1$,  then 
\begin{equation*}
\begin{split}
 |S( xn+tn^2)|^2 &\ll_\eps  N^\eps \sum_{z=1}^{2} \sum_{u=1}^{4N} \min(N,\|z(u t+2x)\|^{-1})\\
  & \leq N^\eps \Big(\sum_{u=1}^{8N} \min(N,\|u t+2x\|^{-1}) +
             \sum_{u=1}^{8N} \min(N,\|u t+4x\|^{-1})
  \Big).
 \end{split}
\end{equation*}

The next lemma provides further marvellous bounds of these sums.
\begin{lemma}[Lemma 9D in \cite{Schmidt}]\label{w-3}
Let $q,a$ be coprime, with $1\leq q\leq N$. Suppose that
\begin{equation}\label{srtong-Dir}
  |\alpha q-a|=\|\alpha q\|<\frac{1}{(1+c)N},
\end{equation}
where $c>1$ is a given constant. Then
\begin{equation}\label{s}
  \sum_{n=1}^{cN} \min(N,\|\alpha n +\beta \|^{-1}) \ll (\log N) \min \Big(\frac{N^2}{q},\frac{N}{\|\beta q \|},\frac{1}{\|\alpha q \|}\Big).
\end{equation}

\end{lemma}
Lemma \ref{w-3}, \eqref{P-1} and  \eqref{q-1}   yield that

\begin{equation}\label{weyl-3}
 P \leq |S( xn+tn^2)|  \leq C  N^{2 \eps} \min \Big(\frac{N}{q^\frac12},\frac{N^\frac12}{\|2x q \|^\frac12} + \frac{N^\frac12}{\|4x q \|^\frac12},\frac{1}{\|t q \|^\frac12}\Big).
\end{equation}
Thus, there exists $j=2$ or $4$ such that
\begin{equation}\label{x-1}
  \|jxq\|\leq N^{1+4\eps} P^{-2}.
\end{equation}

\subsection{Proof of Theorem \ref{T-d}}\label{max-weyl}

For $d\ge2$, let $\alpha\in [\frac{d}{2}+\frac{d}{2(d+1)},d] $ and $\alpha_1+\alpha_2+\cdots+\alpha_d=\alpha$.
For any measurable function $t(x)$ on $\mathbb{T}^d$, 
we consider
$x$ such that
\begin{equation*}
|S(x_{1}n_{1}+t(x)n_{1}^{2})|\!=\! N^{\alpha_1}\!, \! |S(x_{2}n_{2}+t(x)n_{2}^{2})|\!=\! N^{\alpha_2} \!,
\! \cdots \!, \! |S(x_{d}n_{d}+t(x)n_{d}^{2})|\!=\! N^{\alpha_d}.
\end{equation*}

We pigeonhole the regions of $x$ in $\T^d$ into two parts corresponding to two cases.
For the first case in which $x$ satisfies that 
there exists at least one $\alpha_{l_{0}}$ such that $\alpha_{l_{0}}\leq \frac{d}{2(d+1)}+\frac{d}{2}-\frac{d-1}{2d}-\frac{d-1}{2}=\frac12+\frac{1}{2d(d+1)}$, exerting induction on the dimension $d$, by corresponding results on $\mathbb T^{d-1}$, we have
\begin{equation}
 \|u(x,t(x))  \|_{L^\frac{2(d+1)}{d}(\mathbb{T}^d)} \leq N^{\alpha_{l_{0}}} N^{\frac{d-1}{2d}+\frac{d-1}{2}}\leq N^{\frac{d}{2(d+1)} +\frac{d}{2}}.
\end{equation}
For the second case, it suffices to consider $x$ such that $\frac12+\frac{1}{2d(d+1)}<\alpha_1, \cdots, \alpha_d \leq 1.$
According to the analysis above  and  \eqref{q-1}, there exist integers $q, a$ with $(q,a)=1$
such that
\begin{equation}\label{q-2}
0\leq a< q \leq N^{2-2\alpha_l+2\eps}  \text{  and } |qt(x)-a|\leq N^{-2\alpha_l+2\eps}
\end{equation}
for $1\le l\le d$.
Moreover, as per the restriction of $x$ in \eqref{x-1},
there exist $j_1,\cdots,j_d\in\{2,4\},$ such that for $l=1,2,\cdots,d$, we have
\begin{equation}\label{x-2}
  \|j_lx_l q\|\leq2 N^{1-2\alpha_l+4\eps}.
\end{equation}
This of course implies that for $l=1,2,\cdots,d$, there exist $ b_l\in \frac14 \mathbb{Z}\cap [1,4q]$, such that
$$|\beta_l|=|x_l-\tfrac{b_l}{q}|\leq q^{-1} N^{1-2\alpha_l+4\eps}.$$
On the other hand, by the estimate \eqref{weyl-3}, we have for $l=1,2,\cdots,d$,
\begin{equation}
|S( x_{l}n_{l}+t(x)n_{l}^2)|  \leq C  N^{\frac12+ \eps}q^{-\frac12} \min (N^\frac12,|\beta_l|^{-\frac12}).
\end{equation}
This estimate implies
\begin{equation}\label{upper-3}
|u(x,t(x))|=\prod_{l=1}^{d} |S( x_ln_{l}+t(x)n_{l}^2)|  \leq C  N^{\frac{d}{2}+ d \eps}q^{-\frac{d}{2}} \prod_{l=1}^{d} \min (N^\frac12,|\beta_l|^{-\frac12}) .
\end{equation}
Using this bound and defining $\tilde{\alpha}= \frac{1}{2} + \frac{1}{2d(d+1)}$, we obtain
\begin{equation}\label{upper-4}
\begin{aligned}
  &\int_{x\in \mathbb{T}^d(\alpha_1,\cdots,\alpha_d)}|u(x,t(x))|^\frac{2(d+1)}{d} dx \\
 \leq  & C \sum_{1\leq q \leq N^{2-2\tilde{\alpha}+2\eps}} \sum_{b_1\in\frac14 \mathbb{Z}\cap [1,4q]} \cdots \sum_{b_d\in \frac14 \mathbb{Z}\cap [1,4q]} N^{d+1 +C \eps} q^{-(d+1)}  \int_{ |\beta_1|\leq q^{-1} N^{1-2\alpha_1+4\eps}}  \\
 &   \phantom{C \sum_{1\leq q \leq N^{2-2\alpha_d+2\eps}}}\cdots \int_{ |\beta_d|\leq q^{-1} N^{1-2\alpha_d+4\eps}}\prod_{l=1}^{d} \min (N^\frac12,|\beta_l|^{-\frac12})^\frac{2(d+1)}{d} d\beta_1 d\beta_2\cdots d\beta_d\\
\leq  & C \sum_{1\leq q \leq N^{2-2\tilde{\alpha}+2\eps}}   N^{d+1 +C \eps} q^{-1}
\prod_{l=1}^{d} \int_{ |\beta_l|\leq q^{-1} N^{1-2\alpha_l+4\eps}}
 \min(N^\frac12, |\beta_l|^{-\frac12})^\frac{2(d+1)}{d} d\beta_l\\
\leq &  C \sum_{1\leq q \leq N^{2-2\tilde{\alpha}+2\eps}}   N^{d+2+C \eps} q^{-1}.
 \end{aligned}
\end{equation}
Thus, we have
\begin{equation}
\|u(x,t(x))\|_{L^\frac{2(d+1)}{d} ({\mathbb{T}}^d)} \leq C  N^{\frac{d(d+2)}{2(d+1)}+\epsilon}=
 C N^{\frac{d}{2}+\frac{d}{2(d+1)}+\epsilon} ,
\end{equation}
for any $\epsilon>0$.
For the case $1\le p\le\frac{2(d+1)}{d}$, by H\"older's inequality, we have
\begin{equation}
  \Big\|\sup_{0<t<1}|u(x,t)|\Big\|_{L^{p}(\T^d)}
  \le \Big\|\sup_{0<t<1}|u(x,t)|\Big\|_{L^{\frac{2(d+1)}{d}}(\T^d)}
  \lesssim N^{\frac{d}{2}+\frac{d}{2(d+1)}+\epsilon}.
\end{equation}
On the  other hand, for general $p\ge\frac{2(d+1)}{d}$  we have, akin to the estimate \eqref{upper-4},
\begin{equation}
  \Big\|\sup_{0<t<1}|u(x,t)|\Big\|_{L^{p}(\T^d)}\lesssim N^{d-\frac{d}{p}+\epsilon},
\end{equation}
for $p\ge\frac{2(d+1)}{d}$.

\begin{remark}

One may also expect to obtain the results of Theorem \ref{T-d} via the argument of
\cite{A-B} for $\mathbb T$.
However, when one decomposes the Weyl sums as in \cite{A-B} directly, there would be more cross terms for the upper bound of Weyl sums with the increment of the dimension. These cross terms  seem to be more difficult to control, which is the main difference from the situation in $\T$.
 Due to this difference,  if one applies the method in \cite{A-B}, the corresponding result in $\T^d$ would seem unable to reach the sharpness of $s$ for Theorem \ref{T-d} to hold.
\end{remark}

\subsection{Proof of Theorem \ref{T-d-2} }\label{max-weyl-ra}

For $d\ge2$, let $r=(r_1,r_2,\cdots,r_d)\in \mathbb Q^d.$
Similarly, we only consider  $\alpha\in [\frac{d}{2}+\frac{d}{2(d+1)},d] $ and $\alpha_1+\alpha_2+\cdots+\alpha_d=\alpha$.
For any measurable function $t(x)$ on $\mathbb{T}^d$, we consider
$x$ such that
$$|S((x_1-r_1 t(x))n_{1}+t(x)n_{1}^{2})|= N^{\alpha_1}, \cdots , |S((x_d-r_d t(x))n_{d}+t(x)n_{d}^{2})|= N^{\alpha_d}. $$
Similar to the proof of Theorem \ref{T-d} in Subsection \ref{max-weyl},
we also consider two cases for $x$ in $\T^d$.
For the first case in which $x$ satisfies that
there exists at least one $\alpha_{l_{0}}$ such that $\alpha_1\leq \frac{d}{2(d+1)}+\frac{d}{2}-\frac{d-1}{2d}-\frac{d-1}{2}=\frac12+\frac{1}{2d(d+1)}$,  using induction on the dimension $d$, by corresponding results on $\mathbb T^{d-1}$, we have
\[
 \big\| u(x-r t(x),t(x))  \big\|_{L^\frac{2(d+1)}{d}(\mathbb{T}^d)} \leq N^{\alpha_{l_{0}}} N^{\frac{d-1}{2d}+\frac{d-1}{2}}\leq N^{\frac{d}{2(d+1)} +\frac{d}{2}}.
\]
For the second case, it suffices to consider $x$ such that
$\frac12+\frac{1}{2d(d+1)}<\alpha_1, \cdots, \alpha_d \leq 1.$
According to the analysis in Section \ref{preli}, \eqref{P-1} and \eqref{q-1}, there exist integers $q,a$ with $(q,a)=1$ such that
\begin{equation}\label{q-3}
0\leq a< q \leq N^{2-2\alpha_l+2\eps}  \text{  and } |qt(x)-a|\leq N^{-2\alpha_l+2\eps},
\end{equation}
for $1\le l\le d$.
Moreover, by the restriction of $x$ in \eqref{x-1},
there exist $j_1,\cdots,j_d\in\{2,4\},$ such that for $l=1,2,\cdots,d$, 
\begin{equation}\label{x-3}
  \|j_l(x_l- r_l t(x)) q\|\leq2 N^{1-2\alpha_l+4\eps}.
\end{equation}
Let $r_l=\frac{r_{l,1}}{r_{l,2}}$ with $(r_{l,1},r_{l,2})=1.$
From the estimates \eqref{q-3} and \eqref{x-3},
we have
\[
\min( \big|x_l -\tfrac{a}{q} \tfrac{r_{l,1}}{r_{l,2}} -\tfrac{k_l}{2q} \big| ,\big|x_l -\tfrac{a}{q} \tfrac{r_{l,1}}{r_{l,2}} -\tfrac{k_l}{4q} \big|  ) <  \tfrac{2}{q}N^{1-2\alpha_l+4\eps},
\] for some $k_l\in \mathbb Z$.
This estimate implies that for $l=1,2,\cdots,d$, there exist $ b_l\in \frac1{4r_{l,2}} \mathbb{Z}\cap [0,q]$, such that
$$|\beta_l|=|x_l-\tfrac{b_l}{q}| = \min( \big|x_l -\tfrac{a}{q} \tfrac{r_{l,1}}{r_{l,2}} -\tfrac{k_l}{2q} \big| ,\big|x_l -\tfrac{a}{q} \tfrac{r_{l,1}}{r_{l,2}} -\tfrac{k_l}{4q} \big|  )\leq  C q^{-1} N^{1-2\alpha_l+4\eps}.$$

On the other hand, by the estimate \eqref{weyl-3}, we have for $l=1,2,\cdots,d$,
\begin{equation}
|S( (x_l-r_lt(x))n_{l}+t(x)n_{l}^2)|  \leq C  N^{\frac12+ \eps}q^{-\frac12} \min (N^\frac12,|\beta_l|^{-\frac12}),
\end{equation}
which yields
\begin{equation}
\begin{split}
|u(x-rt(x),t(x))|&=\prod_{l=1}^{d} |S((x_l-r_lt(x))n_{l}+t(x)n_{l}^2)| \\
&\leq C  N^{\frac{d}{2}+ d \eps}q^{-\frac{d}{2}} \prod_{l=1}^{d} \min (N^\frac12,|\beta_l|^{-\frac12}) .
\end{split}
\end{equation}
Define $\tilde{\alpha}= \frac{1}{2} + \frac{1}{2d(d+1)}$ and we have
\begin{equation}\label{ration-upper}
\begin{aligned}
  &\int_{x\in \mathbb{T}^d(\alpha_1,\cdots,\alpha_d)}|u(x-rt(x),t(x))|^\frac{2(d+1)}{d} dx \\
 \leq  & C \sum_{1\leq q \leq N^{2-2\tilde{\alpha}+2\eps}} \sum_{b_1\in \frac1{4r_{1,2}} \mathbb{Z}\cap [0,q]} \!\!\cdots\!\!\! \sum_{b_d\in \frac1{4r_{d,2}} \mathbb{Z}\cap [0,q]} \!\!\!N^{d+1 +C \eps} q^{-(d+1)}  \!\int_{ |\beta_1|\leq q^{-1} N^{1-2\alpha_1+4\eps}}  \\
 &   \phantom{C \sum_{1\leq q \leq N^{2-2\alpha_d+2\eps}}}\cdots\int_{ |\beta_d|\leq q^{-1} N^{1-2\alpha_d+4\eps}} \prod_{l=1}^{d} \min (N^\frac12,|\beta_l|^{-\frac12})^\frac{2(d+1)}{d} d\beta_1 d\beta_2\cdots d\beta_d\\
\leq  & C_r \sum_{1\leq q \leq N^{2-2\tilde{\alpha}+2\eps}}   N^{d+1 +C \eps} q^{-1}
\prod_{l=1}^{d} \int_{ |\beta_l|\leq q^{-1} N^{1-2\alpha_l+4\eps}}
 \min(N^\frac12, |\beta_l|^{-\frac12})^\frac{2(d+1)}{d} d\beta_l\\
\leq &  C_r \sum_{1\leq q \leq N^{2-2\tilde{\alpha}+2\eps}}    N^{d+2+C \eps} q^{-1}.
 \end{aligned}
\end{equation}
This estimate implies that
\begin{equation}
\|u(x-rt(x),t(x))\|_{L^\frac{2(d+1)}{d} ({\mathbb{T}}^d)} \leq C  N^{\frac{d(d+2)}{2(d+1)}+\epsilon}=
 C N^{\frac{d}{2}+\frac{d}{2(d+1)}+\epsilon} ,
\end{equation}
for any $\epsilon>0$, where $C$ only depends on $r $ and $\epsilon.$
If $1\le p<\frac{2(d+1)}{d} $, by H\"older's inequality, we have
\begin{equation}
  \|u(x-rt(x),t(x))\|_{L^{p}({\mathbb{T}}^d)}\le \|u(x-rt(x),t(x))\|_{L^{\frac{2(d+1)}{d}} ({\mathbb{T}}^d)} \lesssim N^{\frac{d}{2}+\frac{d}{2(d+1)}+\epsilon},
\end{equation}
for any $\epsilon>0$, where $C$ only depends on $r $ and $\epsilon.$

For general $p\ge\frac{2(d+1)}{d}$, similar to \eqref{ration-upper}, we have
\begin{equation}
\|u(x-rt(x),t(x))\|_{L^p ({\mathbb{T}}^d)} \leq C  N^{\frac{d}{2}+(\frac{d}{2}-\frac{d}{p}+\epsilon)},
\end{equation}
for any $\epsilon>0$, where $C$ only depends on $r $ and $\epsilon.$

\section{Proof of Theorem \ref{T-d-3}}

For the proof of Theorem \ref{T-d-3}, we divide the maximal estimate \eqref{max-irra} on $0<t<1$
into two parts on $0<t<\frac{2}{N}$ and $\frac{2}{N}<t<1$,
for which the maximal estimate on $\mathbb R^d$ and the following Proposition \ref{local-generi} related to the
 Strichartz estimate are exploited respectively.

Recall the Definition \ref{generi} for the genericity of $\beta\in[1,2]^{d-1}$ in Section \ref{intro}
and $Q(n) = |n_1|^2+\beta_{1}|n_2|^{2}+\cdots+\beta_{d}|n_d|^{2}$.
Deng, Germain and Guth in \cite{DGG} established
 pointwise bound of  the Weyl sums
on irrational tori as follows.


\begin{prop}[Proposition 4.6 in \cite{DGG}]\label{local-generi}
  Assume that $\beta=(\beta_2,\cdots,\beta_d)$ is chosen generically. Then for
  $\frac{2}{N} <t <N^K $ for a positive constant $K$, we have
  \begin{equation}\label{DGG-weyl-g}
 |K_N(t,x) |\leq N^{\frac{d+1}{2}+\eps} t^{\frac14},
  \end{equation}
where
$$
K_N(t,x) := \sum_{n\in\Z^d}e^{2\pi i(x\cdot n-tQ(n))}\chi(\frac{n_1}{N})\cdots\chi(\frac{n_d}{N})
$$
and $\chi$ is a smooth, nonnegative function, supported on $B(0,2)$, and equal to $1$ on $B(0,1)$.
\end{prop}

The proof for the maximal estimate \eqref{max-irra} on $\frac{2}{N}<t<1$ is furnished with this proposition. It suffices to consider  the remaining part for the maximal estimate on $0<t<\frac{2}{N}$, which can be solved by the corresponding local-in-time maximal estimate below.

Moyua and Vega in \cite{M-V} proved a local-in-time maximal estimate for the periodic Schr\"odinger equation in $\T$. Here, by similar localization arguments of 
Bourgain-Demeter \cite{BD} (or Hickman \cite{Hickman-2016-2}), we extend their result to that in the higher dimensions for $d\ge2$ and generic Laplacian operator $Q(D)$
 based on the results of Du-Guth-Li in \cite{DGL}
and Du-Zhang in \cite{DZ} in $\R^d.$


%
%
For simplicity, we first consider the two-dimensional case.
\begin{lemma}\label{local-max}
\begin{enumerate}
\item[(1)]
Suppose $u$ is the solution to the Schr\"odinger equation on $\T^{2}$, and $u(x,0) = f(x)$ with \textup{supp}\,$\hat{f}\subset[-N,N]^{2}$. Then for $2\le p\le3$, we have
\begin{equation}\label{local-0}
\Big\|\sup_{0<t<N^{-1}}|u(x,t)|\Big\|_{L^{p}(\T^2)}\lesssim N^{\frac{1}{3}+\epsilon}\|f\|_{L^{2}(\T^2)},
\end{equation}
where $u(x,t) = e^{it\Delta_{\T^2}}f(x)$.

\item[(2)]
Suppose $v$ is the solution to the Schr\"odinger equation on the irrational torus of dimension 2, and $v(x,0)=g(x)$ with \textup{supp}\,$\hat{g}\subset[-N,N]^{2}$.
Then for $2\le p\le3$, we have
\begin{equation}\label{local-tori-0}
  \Big\|\sup_{0<t<\frac{2}{N}}|v(x,t)|\Big\|_{L^{p}(\T^{2})} \lesssim N^{\frac{1}{3}+\epsilon}\|g\|_{L^{2}(\T^2)},
\end{equation}
where $v(x,t)= e^{itQ(D)}g(x)$ and $Q(n)=|n_{1}|^2+\beta_{2}|n_2|^2$ with $\beta_{2}\in[1,2]$.
\end{enumerate}
\end{lemma}

\begin{proof}[\bf Proof]

We divide the proof into two steps as follows. 

Through the usual localization theory (see for example \cite{HV-2016-1}), we establish some equivalent estimates in the former step.
Then, in the latter step, following the arguments of  \cite{Hickman-2016-2},  we deduce the local-in-time maximal estimates of this lemma as a consequence of the first step.

{\bf Step 1.} We first prove the equivalence of the following three estimates:

\begin{equation}\label{local-1}
(a) \,\,\,\quad \quad\quad\quad\quad\Big\|\sup_{0<t<N}|e^{it\Delta}f(x)|\Big\|_{L^{p}(B_{N})}\lesssim N^{\frac{2}{p}-\frac{2}{3}+\epsilon}\|f\|_{L^{2}}
\end{equation}
holds for supp\,$\hat{f}\subset B_{1}\subset\R^2$;

\begin{equation}\label{local-2}
(b) \,\,\,\,\quad \Big\|\sup_{0<t<N}|F(x,t)|\Big\|_{L^{p}(B_{N})} \lesssim N^{\frac{2}{p}-\frac{7}{6}+\epsilon}\|\hat{F}\|_{L^{2}(\mathcal{N}_{N^{-1}}(\mathbb{P}^{2}))}
\end{equation}
holds for supp\,$\hat{F}\subset\mathcal{N}_{N^{-1}}(\mathbb{P}^{2})$, where $\mathbb{P}^{2}=\{(x,|x|^{2}):x\in[-1,1]\}$;

\begin{equation}\label{local-tori-2}
 (c) \,\,\,\,\quad \Big\|\sup_{0<t<N}|R(x,t)|\Big\|_{L^{p}(B_{N})} \lesssim N^{\frac{2}{p}-\frac{7}{6}+\epsilon}\|\hat{R}\|_{L^{2}(\mathcal{N}_{N^{-1}}(\mathbb{E}^{2}))}
\end{equation}
holds for supp\,$\hat{R}\subset\mathcal{N}_{N^{-1}}(\mathbb{E}^2)$, where $\mathbb{E}^{2} = \{(x,Q(x)):x\in B(0,\frac{1}{2})\}$.

First we show \eqref{local-1} implies  \eqref{local-2}.
Since supp\,$\hat{F}\subset\mathcal{N}_{N^{-1}}(\mathbb{P}^{2})$, we have
\begin{equation}
\begin{split}
F(x,t)
&= \int_{|\xi|<1}\int_{\{\tau: |\tau-|\xi|^{2}|<N^{-1}\}}e^{2\pi i(x\cdot\xi+t\tau)}\hat{F}(\xi,\tau)\,\mathrm{d}\xi\mathrm{d}\tau\\
&= \int_{|\tau|<N^{-1}}\int_{|\xi|<1}e^{2\pi i(x\cdot\xi+(\tau+|\xi|^{2})t)}\hat{F}(\xi,\tau+|\xi|^{2})\,\mathrm{d}\xi\mathrm{d}\tau.
\end{split}
\end{equation}
From the maximal estimate \eqref{local-1}, Fubini's theorem and Minkowski's inequality, we have
\begin{align*}
&\Big\|\sup_{0<t<N^{-1}}|F(x,t)|\Big\|_{L^{p}(B_N)}\\
\lesssim & \int_{|\tau|<N^{-1}}\Big\|\sup_{0<t<N}\big|\int_{|\xi|<1}e^{2\pi i(x\cdot\xi+t|\xi|^{2})}\hat{F}(\xi,\tau+|\xi|^{2})\,\mathrm{d}\xi\big|\Big\|_{L^{p}(B_{N})}\mathrm{d}\tau\\
\lesssim & \int_{|\tau|<N^{-1}}N^{\frac{2}{p}-\frac{2}{3}+\epsilon}\|\hat{F}(\xi,\tau+|\xi|^{2})\|_{L^{2}_{\xi}(|\xi|<1)}\mathrm{d}\tau\\
\lesssim & N^{\frac{2}{p}-\frac{7}{6}+\epsilon}\|\hat{F}\|_{L^{2}(\mathcal{N}_{N^{-1}}(\mathbb{P}^{2}))}.
\end{align*}

Next, we prove \eqref{local-2} implies \eqref{local-tori-2}.
Let $A=\{(\xi_1,\sqrt{\beta_2}\xi_2):\xi\in B(0,\frac{1}{2})\}$. By the fact that $\beta_{2}\in[1,2]$, 
it is easy to see that
 $A\in B(0,1)$.
Since supp\,$\hat{F}\subset\mathcal{N}_{N^{-1}}(\mathbb{E}^2)$, using change of variables, we have
\begin{equation}\label{R-formula}
\begin{split}
  R(x,t)
  &=\int_{|\xi|<\frac{1}{2}}\int_{\{\tau:|\tau-Q(\xi)|<N^{-1}\}}e^{2\pi i(x\cdot\xi+t\tau)}\hat{R}(\xi,\tau)\,\mathrm{d}\xi\mathrm{d}\tau\\
  &=\int_{|\tau|<N^{-1}}\int_{|\xi|<\frac{1}{2}}e^{2\pi i(x\cdot\xi+t(\tau+Q(\xi)))}\hat{R}(\xi,\tau+Q(\xi))\,\mathrm{d}\xi\mathrm{d}\tau\\
  &=\frac{1}{\sqrt{\beta_{2}}}
    \int_{|\tau|<N^{-1}}\int_{A}e^{2\pi i(\tilde{x}\cdot\xi+t(\tau+|\xi|^{2}))}\hat{R}(\tilde{\xi},\tau+|\xi|^{2})\,\mathrm{d}\xi\mathrm{d}\tau\\
  &=\frac{1}{\sqrt{\beta_{2}}}
    \int_{|\tau|<N^{-1}}\int_{|\xi|<1}e^{2\pi i(\tilde{x}\cdot\xi+t(\tau+|\xi|^{2}))}\hat{R}(\tilde{\xi},\tau+|\xi|^{2})
    \chi_{A}(\xi)\,\mathrm{d}\xi\mathrm{d}\tau,
\end{split}
\end{equation}
where $\tilde{x} =(x_{1},\sqrt{\beta_{2}}^{-1}x_{2})$ and $\tilde{\xi} =(\xi_{1},\sqrt{\beta_{2}}^{-1}\xi_{2})$. 
Let $\hat{\tilde{R}}(\xi,\tau+|\xi|^{2}) = \hat{R}(\tilde{\xi},\tau+|\xi|^{2})\chi_{A}(\xi)$. By the representation formula \eqref{R-formula} of $R(x,t)$  and the support of $\hat{R}$, it is easy to see that supp\,$\hat{\tilde{R}}(\xi,\tau)\subset\mathcal{N}_{N^{-1}}(\mathbb{P}^{2})$ .
Then by the estimate \eqref{local-2}, we have
\begin{equation}
  \begin{split}
    &\Big\|\sup_{0<t<N}|R(x,t)|\Big\|_{L^{p}(B_{N})}\\
   \lesssim  & N^{\frac{2}{p}-\frac{7}{6}+\epsilon}\|\hat{\tilde{R}}(\xi,\tau)\|_{L^{2}(\mathcal{N}_{N^{-1}}(\mathbb{P}^{2}))}\\
   \approx  & N^{\frac{2}{p}-\frac{7}{6}+\epsilon}\|\hat{R}(\tilde{\xi},\tau)\chi_{A}(\xi)\|_{L^{2}(\mathcal{N}_{N^{-1}}(\mathbb{P}^{2}))}\\
   \lesssim  & N^{\frac{2}{p}-\frac{7}{6}+\epsilon}\|\hat{R}\|_{L^{2}(\mathcal{N}_{N^{-1}}(\mathbb{E}^{2}))}.
  \end{split}
\end{equation}

Finally, we  prove \eqref{local-1} by \eqref{local-tori-2}.
We can rewrite $e^{it\Delta}f(x)$ with supp\,$\hat{f}\subset B(0,1)\subset\R^2$ as follows,
\begin{equation}
  \begin{split}
    e^{it\Delta}f(x)
    & = \int e^{ix\cdot\xi}e^{-it|\xi|^2}\hat{f}(\xi)\,\mathrm{d}\xi\\
    & = 4\sqrt{\beta_2}\int e^{ix\cdot(2\xi_1,2\sqrt{\beta_2}\xi_2)}e^{-i4tQ(\xi)}\hat{f}(2\xi_1,2\sqrt{\beta_2}\xi_2)\,\mathrm{d}\xi\\
    & = 4\sqrt{\beta_2}\int e^{i2\tilde{x}\cdot\xi}e^{-i4tQ(\xi)}\hat{h}(\xi)\,\mathrm{d}\xi\\
    & = 4\sqrt{\beta_2}e^{i4tQ(D)}h(2\tilde{x}),
  \end{split}
\end{equation}
where $\tilde{x}=(x_1,\sqrt{\beta_{2}}x_2)$ and $\hat{h}(\xi) = \hat{f}(2\xi_1,2\sqrt{\beta_2}\xi_2)$
with supp\,$\hat{h}\subset\{\xi:(\xi_1,\sqrt{\beta_2}\xi_2)\in B(0,\frac{1}{2})\}\subset B(0,\frac{1}{2})$.
Then, by the translation invariance of the operator $e^{it\Delta}$ and change of variables, we have
\begin{equation}
  \begin{split}
    \Big\|\sup_{0<t<N}|e^{it\Delta}f(x)|\Big\|_{L^{p}(B_{N})}
    & \lesssim \Big\|\sup_{0<t<N}|e^{it\Delta}f(x)|\Big\|_{L^{p}(B_{\frac{N}{4}})}\\
    & \lesssim \Big\|\sup_{0<t<N}|e^{i4tQ(D)}h(2\tilde{x})|\Big\|_{L^{p}(B_{\frac{N}{4}})}\\
    & \lesssim \Big\|\sup_{0<t<4N}|e^{itQ(D)}h(x)|\Big\|_{L^{p}(B_{N})}\\
    & \lesssim \Big\|\sup_{0<t<\tilde{N}}|e^{itQ(D)}h(x)|\Big\|_{L^{p}(B_{\tilde{N}})},
  \end{split}
\end{equation}
where $\tilde{N} = 4N$.

Choose $\varphi(x,t)\in\mathcal{S}(\R^{3})$ with Fourier support in $B^{3}_{1}$
such that $|\varphi(x,t)|\sim 1$ on $B^{3}_{1}$.
Let $\varphi_{N}(x,t) := \varphi(\tfrac{(x,t)}{N}) $ with supp\,$\hat{\varphi}_{N}\subset B^{3}_{N^{-1}}$.
Then it is easy to see supp\,$\mathcal{F}_{x,t}(e^{itQ(D)}h(x)\cdot
\varphi_{\tilde{N}}(x,t))\subset\mathcal{N}_{\tilde{N}^{-1}}(\mathbb{E}^{2})$.
Through a direct computation,
we have
\begin{equation}
  \mathcal{F}_{x,t}(e^{itQ(D)}h(x)\cdot\varphi_{\tilde{N}}(x,t))(\xi,\tau)
   = \int_{B(\xi,\tilde{N}^{-1})}(\mathcal{F}_{x,t}
   \varphi_{\tilde{N}})(\xi-y,\tau+|y|^2)\hat{h}(y)\,\mathrm{d}y.
\end{equation}
By H\"older's inequality and the properties of supp\,$\mathcal{F}_{x,t}(\varphi_{\tilde{N}})$
and supp\,$\mathcal{F}_{x,t}(e^{itQ(D)}h(x)\cdot\varphi_{\tilde{N}}(x,t))$, we obtain
\begin{align*}
  &|\mathcal{F}_{x,t}(e^{itQ(D)}h(x)\cdot\varphi_{\tilde{N}}(x,t))(\xi,\tau)|\\
\le&\int_{B(\xi,\tilde{N}^{-1})}|(\mathcal{F}_{x,t}\varphi_{\tilde{N}})(\xi-y,\tau+|y|^2)\hat{h}(y)|\,\mathrm{d}y\\
  \le&\tilde{N}^{-1}\|(\mathcal{F}_{x,t}\varphi_{\tilde{N}})(\xi-y,\tau+|y|^2)\hat{h}(y)\|_{L^{2}_{y}}.
\end{align*}

Thus, by the estimate \eqref{local-tori-2}, $\tilde{N}=4N$, Fubini's theorem and Plancherel's theorem, we have
\begin{align*}
  &\Big\|\sup_{0<t<N}|e^{it\Delta}f(x)|\Big\|_{L^{p}(B_N)}\\
  \lesssim  & \Big\|\sup_{0<t<\tilde{N}}|e^{itQ(D)}h(x)|\Big\|_{L^{p}(B_{\tilde{N}})}\\
  \lesssim  & \tilde{N}^{\frac{2}{p}-\frac{7}{6}+\epsilon}\|\mathcal{F}_{x,t}(e^{itQ(D)}h(x)
    \cdot\varphi_{\tilde{N}}(x,t))(\xi,\tau)\|_{L^{2}(\mathcal{N}_{\tilde{N}}(\mathbb{E}))}\\
  \lesssim  & \tilde{N}^{\frac{2}{p}-\frac{13}{6}+\epsilon}\|(\mathcal{F}_{x,t}\varphi_{\tilde{N}})(\xi-y,\tau+|y|^2)\hat{h}(y)
    \|_{L^{2}_{\xi,\tau,y}(\mathcal{N}_{\tilde{N}}(\mathbb{E}))}\\
  \lesssim  & \tilde{N}^{\frac{2}{p}-\frac{13}{6}+\epsilon}\|\varphi_{\tilde{N}}(x,t)\|_{L^{2}_{x,t}}\|h\|_{L^{2}}\\
  \lesssim  & \tilde{N}^{\frac{2}{p}-\frac{2}{3}+\epsilon}\|\hat{f}(2\xi,2\sqrt{\beta_2}\xi_2)\|_{L^{2}}\\
  \lesssim  & N^{\frac{2}{p}-\frac{2}{3}+\epsilon}\|f\|_{L^{2}}.
\end{align*}

{\bf Step 2.} Utilizing the result in step 1, we prove that if \eqref{local-2} holds, then 
\begin{equation}\label{discret-1}
  \Big\|\sup_{0<t<N^{-1}}\big|\sum_{1\le n_1,n_2\le N}a_{n}e^{2\pi i(x\cdot n +|n|^{2}t)}\big| \Big\|_{L^{p}(B^{2}_{1})}
  \le N^{\frac{1}{3}+\epsilon}\Big(\sum_{1\le n_1,n_2\le N}|a_{n}|^{2}\Big)^{\frac{1}{2}}.
\end{equation}
Also, we can see that if \eqref{local-tori-2} holds, then 
\begin{equation}\label{discret-tori-1}
  \Big\|\sup_{0<t<N^{-1}} \big|\sum_{1\le n_{1},n_{2}\le N}a_{n}e^{2\pi i(x\cdot n+tQ(n))}\big|\Big\|_{L^{p}(B^{2}_{1})}
  \lesssim N^{\frac{1}{3}+\epsilon}\Big(\sum_{1\le n_{1},n_{2}\le N}|a_n|^{2}\Big)^{\frac{1}{2}}.
\end{equation}
Since the proof of \eqref{discret-tori-1} by \eqref{local-tori-2} is similar to that of \eqref{discret-1} by \eqref{local-2},  we just present  the proof of the former estimate \eqref{discret-1}.

Let $\Lambda=N^{-1}\Z^{2}\cap B^{2}(0,1)$ and $G(x,t) = \sum_{n\in\Lambda}a_{n}e^{2\pi i(x\cdot n+|n|^{2}t)}$.
Since $\widehat{G\varphi_{N}}$ is supported in $\mathcal{N}_{N^{-1}}(\mathbb{P}^{2})$, by the
estimate \eqref{local-2}, we have
\begin{equation}\label{S-1}
\Big\|\sup_{0<t<N}|G(x,t)\varphi_{N}| \Big\|_{L^{p}(B^{2}_{N})}
\lesssim N^{\frac{2}{p}-\frac{7}{6}+\epsilon}\|\widehat{G\varphi_N}\|_{L^{2}(\mathcal{N}_{N^{-1}}(\mathbb{P}^{2}))}.
\end{equation}
By rescaling,
\begin{equation}\label{S-1-2}
\begin{split}
 \text{LHS of \eqref{S-1}}
\gtrsim &N^{\frac{2}{p}}\Big\|\sup_{0<t<N^{-1}}\big|\sum_{1\le n_1,n_2\le N}a_{n}e^{2\pi i(n\cdot x+|n|^{2}t)}\big|\Big\|_{L^{p}(B^{2}_{1})}.
\end{split}
\end{equation}
On the other hand, through a direct calculation, we have
\begin{equation}\label{S-1-1}
\begin{split}
  \|\widehat{G\varphi_N}\|_{L^{2}(\mathcal{N}_{N^{-1}}(\mathbb{P}^{2}))}
&\le\Big(\sum_{n\in\Lambda}|a_{n}|^{2}\int|\hat{\varphi}_{N}(\xi-n,\tau-|n|^{2})|^{2}\,\mathrm{d}\xi\mathrm{d}\tau\Big)^{\frac{1}{2}}\\
&\lesssim N^{\frac{3}{2}}\Big(\sum_{n\in\Lambda}|a_{n}|^{2}\Big)^{\frac{1}{2}}.
\end{split}
\end{equation}
Combining \eqref{S-1}, \eqref{S-1-2} and \eqref{S-1-2} together, we obtain \eqref{discret-1}.
This completes the proof.
\end{proof}

In higher dimensions $d\ge3$, as leading to the proof as in Lemma \ref{local-max}, we can also
establish the local-in-time maximal estimates for the operators $e^{it\Delta}$ and $e^{itQ(D)}$
as follows.

\begin{cor}\label{d-dim-1}

  For $d\ge3$ and $\epsilon>0$, we have
  \begin{equation}
  \Big\|\sup_{0<t<N^{-1}}|u(x,t)|\Big\|_{L^{2}(\T^d)} \lesssim_\epsilon N^{\frac{d}{2(d+1)}+\epsilon}\|f\|_{L^{2}(\T^d)},
  \end{equation}
  where $u(x,t)= e^{it\Delta}f(x)$ with \textup{supp}\,$\hat{f}\subset[-N,N]^d$;
  and
  \begin{equation}\label{max-gene-d}
  \Big\|\sup_{0<t<\frac{2}{N}}|v(x,t)|\Big\|_{L^{2}(\T^d)}\lesssim_\epsilon N^{\frac{d}{2(d+1)}+\epsilon}\|g\|_{L^{2}(\T^d)},
  \end{equation}
  where $v(x,t) = e^{itQ(D)}g(x)$ with \textup{supp}\,$\hat{g}(\xi)\subset[-N,N]^d$ and $Q(n)=|n_1|^2+\beta_2|n_2|^{2}+\cdots+\beta_{d}|n_d|^{2}$
  with $\beta_{j}\in[1,2]$ for $2\le j\le d$.

\end{cor}

\begin{remark}
  The proofs of Lemma \ref{local-max} and Corollary \ref{d-dim-1} rely on the following two maximal
  estimates for the Schr\"odinger operator in $\R^d$.

In $\R^2$, Du-Guth-Li in \cite{DGL} 
  obtained that if $2\le p\le3$, then
  \begin{equation}\label{D-1}
    \Big\|\sup_{0<t\le R}|e^{it\Delta}f|\Big\|_{L^{p}(B(0,R))}\lesssim_{\epsilon}R^{\frac{2}{p}-\frac{2}{3}+\epsilon}\|f\|_{L^{2}(\R^2)}
  \end{equation}
  holds for all $R\ge1$ and all $f\in L^{2}(\R^{2})$ with \textup{supp}\,$\hat{f}\subset B(0,1)$, which is exactly \eqref{local-1}.
  Later in $\R^d$ with $d\ge3$, Du-Zhang in \cite{DZ} 
  showed that
  \begin{equation}\label{D-2}
    \Big\|\sup_{0<t\le R}|e^{it\Delta}f|\Big\|_{L^{2}(B(0,R))}\lesssim_{\epsilon}R^{\frac{d}{2(d+1)}+\epsilon}\|f\|_{L^{2}(\R^d)}
  \end{equation}
  holds for all $R\ge1$ and all $f\in L^{2}(\R^{d})$ with \textup{supp}\,$\hat{f}\subset\{\xi\in\R^d: |\xi|\sim1\}$.


From the estimates \eqref{D-1} and \eqref{D-2}, we can see 
  why the parameter $p$ of $L^{p}$ norm for the maximal estimates in Lemma \ref{local-max} and Corollary \ref{d-dim-1} differ from each other.

\end{remark}

{\bf Proof of Theorem \ref{T-d-3}.}
Proposition \ref{local-generi}, Lemma \ref{local-max} (2) and Corollary \ref{d-dim-1} 
($d$-dimensional version) would imply that for generic $\beta=(\beta_2,\cdots,\beta_d)$ and $Q(n)=|n_1|^2+ \beta_2 |n_2|^2+\cdots \beta_d|n_d|^2 $,
\begin{equation}\label{m-w-g}
 \Big\| \sup_{0<t<1} \big|\sum_{n\in [-N,N]^d \cap \mathbb Z^d} e^{2\pi i(x\cdot n +tQ(n))}\big|  \Big\|_{L^2(\mathbb T^d)}
 \lesssim_\eps N^{\frac{d+1}{2}+\eps},
\end{equation}
for any $\eps>0$.
In fact, according to estimates \eqref{DGG-weyl-g}, \eqref{local-tori-0} and \eqref{max-gene-d}, we have
\begin{align*}
  &\Big\| \sup_{0<t<1} \big|\sum_{n\in [-N,N]^d \cap \mathbb Z^d} e^{2\pi i(x\cdot n +tQ(n))}\big|\Big\|_{L^2(\mathbb T^d)}\\
  \le &\Big\| \sup_{0<t<\frac{2}{N}} \big|\sum_{n\in [-N,N]^d \cap \mathbb Z^d} e^{2\pi i(x\cdot n +tQ(n))}\big|\Big\|_{L^2(\mathbb T^d)}\\
  &+\Big\| \sup_{\frac{2}{N}<t<1} \big|\sum_{n\in [-N,N]^d \cap \mathbb Z^d} e^{2\pi i(x\cdot n +tQ(n))}\big|\Big\|_{L^2(\mathbb T^d)}\\
  \lesssim & N^{\frac{d}{2}+\frac{d}{2(d+1)}+\epsilon} +N^{\frac{d+1}{2}+\epsilon}\\
  \lesssim & N^{\frac{d+1}{2}+\epsilon}.
\end{align*}
This completes the proof. 



\section{Lower bounds of the maximal function of  Weyl sums}

In this section, we consider the necessity of the exponent $s_{p}$ for the maximal estimate \eqref{max-Lp}
to hold in Theorem \ref{T-d}.

\subsection{Lower bounds for Theorem \ref{T-d}}
{\bf Case 1, $p\ge\frac{2(d+1)}{d}$.}

In this case, we consider a simple example in higher dimension, which is similar to that in one dimension
given by Barron in \cite{A-B} and indicates the occurrence of constructive interference for
$\sup_{0<t<1}|\omega_{N}(x,t)|$ on some small region of $\T^d$.

Take $E=[0,10^{-6}N^{-1}]^d$. Then for $x\in E$, we have
\begin{equation}
\sup\limits_{0<t<1}|\omega_{N}(x,t)|\gtrsim N^d,
\end{equation}
which yields that
\begin{equation}\label{max-lower}
\Big\|\sup_{0<t<1}|\omega_{N}(x,t)| \Big\|_{L^{p}(\T^{d})}
\gtrsim \Big\|\sup_{0<t<1}|\omega_{N}(x,t)|\Big\|_{L^{p}(E)}
\gtrsim N^{d-\frac{d}{p}}=N^{\frac{d}{2}+\big(\frac{d}{2}-\frac{d}{p}\big)}.
\end{equation}

\begin{remark}
When $p=\frac{2(d+2)}{d}$, the exponent $\frac{d}{2}-\frac{d}{p}=\frac{d}{d+2}$  in the estimate
\eqref{max-lower} is consistent with  Proposition 3.1 of \cite{CLS-2021},
in which by Strichartz estimates on $\mathbb{T}^d$
Compaan, Luc\`a and Staffilani obtained
that the maximal estimate
\begin{equation}
  \Big\|\sup_{0<t<1}|e^{it\Delta}f(x)|\Big\|_{L^{\frac{2(d+2)}{d}}(\T^d)} \lesssim \|f\|_{H^{s}(\T^d)}
\end{equation}
holds for $s>\frac{d}{d+2}$, and fails for $s<\frac{d}{2(d+1)}$.
\end{remark}

{\bf Case 2, $1\le p<\frac{2(d+1)}{d}$.}
In this case, we study the lower bounds of \eqref{max-Lp} in Theorem \ref{T-d} via analysing
the major arc of Weyl sums in $d$-dimensions.
Inspired by the simultaneous Dirichlet approximation in Bourgain \cite{Bourgain-2} and \cite{Pierce},
we gain the following proposition.


\begin{prop}\label{L^1-prop}
  Fix $N>1$ and suppose $1\le q\le N^{\frac{d}{d+1}}$ is an  integer with $q\equiv0 (\,  \text{\, mod } \,4\, )$. Suppose $b$ is an even integer, $1\le a,b<q$ and $(a,q)=1$.
For $|x-\frac{b}{q}|\le\frac{1}{100N}$ , $x\in\T$ and $|t-\frac{a}{q}|\le\frac{1}{100N^{2}}$,
we have
  \begin{equation}
    \left|\sum_{n=1}^{N}e^{2\pi i(nx+n^{2}t)}\right| \ge c\frac{N}{q^{\frac{1}{2}}}.
  \end{equation}
\end{prop}

\begin{proof}

Utilizing the  partial summation Lemma 5.1 of \cite{Pierce} and taking $f(n)=n\frac{b}{q}+n^2\frac{a}{q}$ and $h(n)=(x-\frac{b}{q})n+(t-\frac{a}{q})n^2$,
we have,
\[\begin{split}
 \left|\sum_{n=1}^{N}e^{2\pi i(nx+n^{2}t)}\right| \geq &
\left|\sum_{n=1}^{N}e^{2\pi if(n)}\right| -\! \sup_{k\in
[1,N]}\!\left|\sum_{n=1}^{k}  e^{2\pi if(n)} \right| \cdot
\sup_{n\in[1,N]}\big|h'(n)\big|N \\
:= & \quad W_1-W_2.
\end{split}
\]
By Lemma 3.1 and Lemma 3.2 of \cite{Pierce} we have  $W_1 \geq  \frac{1}{4} \frac{N}{q^\frac12}$. On the other hand, we have $W_2 \leq \frac{1}{100}   \frac{N}{q^\frac12} $ by our assumptions on $(x,t)$ and
Lemma 3.2 of \cite{Pierce}. Thus, we finish the proof.
\end{proof}

Based on this proposition, we define the major arc as
\[
\mathcal{M}(q,a_1,\cdots,a_d):= \Big\{x\in\mathbb T^d: |x_j-\tfrac{a_j}{q}|\leq \tfrac{1}{100 N}, 1\leq j \leq N \Big\}.
\]
Let \[E=\bigsqcup_{1\leq q \leq N^{d/(d+1)},  q\equiv \, 0( \text{mod}\, 4\,);
\atop 1\leq a_1, \cdots, a_d \leq q, \text{\quad even\quad }} \mathcal{M}(q,a_1,\cdot,a_d).\]
In conformity with Proposition \ref{L^1-prop}, for $x\in E $, there exists $t=t(x)$ such that
\begin{equation}\label{lower-1}
|u(x,t(x))|=\Big|\prod_{1\leq j \leq d} \sum_{n_j=1}^{N}e^{2\pi i(n_{j}x_{j}+n_{j}^{2}t(x))} \Big|
 \gtrsim N^{\frac{d}{2}+\frac{d}{2(d+1)}}.
\end{equation}
We just need to show $|E|\gtrsim 1$, 
which one can derive from 
adapting the arguments of Pierce  \cite{Pierce}.
In fact, by  the simultaneous Dirichlet approximation, we have
\[E_1:=\bigsqcup_{1\leq q \leq  \tfrac14 N^{d/(d+1)},
\atop
1\leq a_1, \cdots, a_d \leq   q, }
 \big\{x\in\mathbb T^d:
 |x_j-\tfrac{a_j}{q}|\leq  \tfrac{4}{ q N^\frac{1}{d+1}},
 1\leq j \leq N \big\} =\mathbb T^d.\]
We rescale $E_1$ by denoting
 \[E_2 := \bigsqcup_{1\leq q \leq N^{d/(d+1)},  q\equiv \, 0( \text{mod}\, 4\,);
\atop 1\leq a_1, \cdots, a_d \leq q, \text{\quad even\quad } }
 \Big\{x\in\mathbb T^d:
 |x_j-\tfrac{a_j}{q}|\leq \tfrac{1}{ q N^\frac{1}{d+1}},
 1\leq j \leq N \Big\}. \]
 Since all the  $\frac14$-scaled intervals which comprise $E_1$ are contained  in $E_2$, by Vitali's covering lemma or Lemma 5.2 of \cite{Pierce},  we have  $|E_2| \gtrsim |E_1|=1.$

Moreover, for $0<c_1\ll1$, let
\[E_3 :=  \bigsqcup_{c_1 N^{d/(d+1)} \leq q \leq N^{d/(d+1)},  q\equiv \, 0( \text{mod}\, 4\,);
\atop 1\leq a_1, \cdots, a_d \leq q, \text{\quad even\quad }} \big\{x\in\mathbb T^d: |x_j-\tfrac{a_j}{q}|\leq \tfrac{1}{ q N^\frac{1}{d+1}}, 1\leq j \leq N \big\},\]
then we have
\[
|E_2\setminus E_3|\leq \sum_{1\leq q\leq c_1 N^{\frac{d}{d+1}} }  \sum_{a_1,\cdots, a_d} \tfrac{1}{ q^d N^\frac{d}{d+1}} \leq c_1.
\]
Thus, by taking $c_1$ small enough, we have $|E_3| \gtrsim 1.$
Analogously, by Vitali's covering lemma again, we show that
\[E_4 := \bigsqcup_{c_1 N^{d/(d+1)} \leq q \leq N^{d/(d+1)},  q\equiv \, 0( \text{mod}\, 4\,);
\atop 1\leq a_1, \cdots, a_d \leq q, \text{\quad even\quad }} \Big\{x\in\mathbb T^d: |x_j-\tfrac{a_j}{q}|\leq \tfrac{1}{100N}, 1\leq j \leq N \Big\}\]
has positive measure.
Hence
\begin{equation}\label{lower-2}
|E| \geq|E_4| \gtrsim 1.
\end{equation}
In conclusion, from the estimates \eqref{lower-1} and \eqref{lower-2}, we obtain
\begin{equation}
\Big\|\sup_{0<t<1}|u(x,t)|\Big\|_{L^{p}(\T^d)} \ge \Big\|\sup_{0<t<1}|u(x,t)|\Big\|_{L^{1}(E)} \gtrsim N^{\frac{d}{2}+\frac{d}{2(d+1)}}
\end{equation}
for $1\le p<\frac{d}{2(d+1)}$.

\subsection{Lower bounds for  $L^p$ maximal estimates for certain functions}\label{lower-bound-for-certain-functions}

For general functions, as in \cite{DKWZ-2020}, one may  conjecture \eqref{max-Lp} does not hold for large $p$ and $d\geq 3$.
Here we consider the $L^{p}$ maximal estimates for certain special functions by adapting the ideas of \cite{CLS-2021} and \cite{DKWZ-2020},
which give the lower bounds as the Weyl sums.


Let $x=(x',x'')=\T^m\times \T^{d-m}$, for $1\leq m \leq d-1.$
Define $D=\lfloor N^{1-\kappa} \rfloor,$ for $\kappa\in(0,1)$.
We take
        $$f(x) :=  f_1(x')f_2(x''):=  \prod_{1\leq j \leq m}\sum_{1\leq n_j\leq N} e^{2\pi i x_j n_j} \times    \prod_{m+1\leq j \leq d}  e^{2\pi i  x_j} \sum_{1\leq n_j\leq N/D} e^{2\pi i  x_j D n_j} .$$
Define
\[E_m :=\bigsqcup_{c_1 N^{m/(m+1) }\leq q \leq N^{m/(m+1)},  q\equiv \, 0( \text{mod}\, 4\,);
\atop 1\leq a_1, \cdots a_m \leq q, \text{\quad even\quad }} \mathcal{M}(q,a_1,\cdots,a_m).\]
By the previous analysis, 
the set $E_{m} \subset \mathbb T^m$  has positive measure, and such that
\[
\sup_{0<t<1}|e^{it\Delta_{\mathbb T^m}}f_1(x')| \geq N^{\frac{m}{2(m+1)}}, \quad\text{ for  }\,\, x\in
E_{m}.
\]
In fact, we can choose $t=\frac{a}{q}$ with $(a,q)=1$ for $x' \in \mathcal{M}(q,a_1,\cdots,a_m)$.

On the other hand, by letting
\[
g(x_j):=\sum_{1\leq n_j\leq N/D} e^{2\pi i  x_j D n_j},
\]
we have
\[
|e^{it\Delta_{\mathbb T}}g(x_j)|=\sum_{1\leq n_j\leq N/D} e^{2\pi i ( x_j D n_j+ |n_j|^2 D^2 t )}.
\]
By Lemma 5.1 of \cite{Pierce} with
\begin{equation*}\left\{\begin{aligned}
&f(n_{j})=0, \;\text{and}\; h(n_j)=x_jD n_j+|n_j|^2(D^2\frac{a}{q}-k),\\
 &h'(n_j)=x_jD  +2n_j(D^2\frac{a}{q}-k),\text{for\, some}\;\, k\;\,\text{such \,that}\\
  &|D^2 \frac{a}{q}-k|\leq \frac{1}{q},\end{aligned}\right.\end{equation*}
   we have
\[\begin{split}
|e^{i\frac{a}{q}\Delta_{\mathbb T}}g(x_j)|\geq &~ N/D - N/D |h'(n_j)| \frac{N}{D}\\
      \geq  &   ~ N/D\Big[1- \big( \frac{D}{100N}+ \frac{N}{D} \frac{1}{q}\big) \frac{N}{D} \Big]\\
      \gtrsim & ~ \frac{N}{D}
  \end{split}
\]
if $D\gg \frac{N}{q^\frac12} \sim N^{\frac{m+2}{2(m+1)}}.$ This estimate implies $\kappa\in(0,\frac{m}{2(m+1)})$.

Thus, by the fact that $e^{it\Delta_{\mathbb T}} (e^{i\cdot} \phi(\cdot)) (x) =e^{i(x-t)}\big(e^{it\Delta} \phi\big)(x-2t)$,  we have
\[
\big| e^{it\Delta_{\mathbb T^{d-m}}} f_2(x'') \big| \gtrsim \frac{N^{d-m}}{D^{d-m}} \sim N^{\kappa (d-m)}
\]
on the region
\[
F_q^{d-m}:= \cup_{1\leq a\leq q \atop
(a,q)=1} B(\tfrac{2a}{q}, \tfrac{1}{100N})^{d-m},
\]
which  satisfies $|F_q^{d-m}| \geq \frac{q}{N^{d-m}}\sim N^{\frac{m}{m+1}-(d-m)}$.
Hence, we have for $\kappa\in(0,\frac{m}{2(m+1)})$,
\[
\begin{split}
\frac{\|e^{it(x) \Delta} f(x)\|_{L^p}}{\|f\|_2} \gtrsim & N^{\frac{m}{2(m+1)}  } N^{\kappa(d-m)} N^{\frac{m}{p(m+1)}-\frac{d-m}{p}} N^{-\frac{\kappa}{2}(d-m)}
\\
= & N^{\frac{m}{2(m+1)}  } N^{\frac12\kappa(d-m)} N^{\frac{m}{p(m+1)}-\frac{d-m}{p}}  .
\end{split}
\]

As a example, if we take $m=2,d=3,p=\frac{8}{3}, \kappa=1$,
the above bound behaves as $\frac{1}{3}+\frac{1}{6}-\frac{1}{8}=\frac38.$
If we take $m=2,d=3,p=\frac{10}{3}, \kappa=1$,
the above bound behaves as $\frac{3}{5}>\frac38.$

\section{Applications of the maximal estimate in Theorem \ref{T-d} }\label{Appli}


In this section, we give some applications of Theorem \ref{T-d},
including the following estimate for the Lebesgue measure of the set where the maximal Weyl sums gain large values.
 Before that, we introduce 
 an efficacious lemma, which extends the locally
constant property of Weyl sums in one dimension obtained by Barron in \cite{A-B} to the higher dimensions. 

For $0<\alpha<d$,
we tile $\T^{d+1}$ by axis-parallel rectangles of dimensions
$(N^{-(d+1)+\alpha})^{d}\times N^{-2(d+1)+2\alpha}$
and denote the set of these rectangles by $\Lambda_{\alpha}$.

\begin{defi}\label{2-dimensional}
Let $\pi_{x}:[0,1]^{d}\times[0,1] \to [0,1]^{d}$ be a projection in the $x$ variables. We say $\mathcal{Q}\subset\Lambda_{\alpha}$ is d-dimensional at scale $(N,\alpha)$,
if for $\forall x_0\in[0,1]^{d}$, there exist at most 2 rectangles $Q\in\mathcal{Q}$ such that $x\in\pi_{x}(Q)$.
\end{defi}

The following lemma is the locally constant property for the higher dimensional version of the quadratic Weyl sums.

\begin{lemma}\label{local-constant}
Let  $\eta>0$ be a small constant.
Assume $cN^{\alpha}\le |\omega_{N}(x_0,t_0)| <CN^{\alpha}$ for some $(x_0,t_0)\in[0,1]^{d}\times[0,1]$  and $0<\alpha<d$.
Let $Q$ be a rectangle of dimensions $(N^{-(d+1)+\alpha-\eta})^{d}\times N^{-2(d+1)+2\alpha-\eta}$ which contains $(x_0,t_0)$.
Then we have $|\omega_{N}(x,t)|\sim N^{\alpha}$ for $\forall (x,t)\in Q$.
\end{lemma}

\begin{proof}[\bf Proof] 
By a simple mean value theorem, we have
\begin{align*}
~&|\omega_{N}(x,t)-\omega_{N}(x_0,t_0)|\\
=~&\Big|\sum_{\substack{1\le n_{j}\le N,\\1\le j\le d}}e^{2\pi i(x\cdot n+|n|^{2}t)}
- \sum_{\substack{1\le n_{j}\le N,\\1\le j\le d}}e^{2\pi i(n\cdot x_{0}+|n|^{2}t_{0})}\Big|\\
\lesssim ~& \sum_{\substack{1\le n_{j}\le N,\\1 \le j\le d}}(|n|\cdot|x-x_{0}|+|n|^{2}|t-t_0|)\\
\lesssim~& \sum_{\substack{1\le n_{j}\le N,\\1\le j\le d}}|n|N^{-(d+1)+\alpha-\eta}
  + \sum_{\substack{1\le n_{j}\le N,\\1\le j\le d}}|n|^{2}N^{-2(d+1)+2\alpha-\eta}\\
\lesssim~& N^{\alpha-\eta}.
\end{align*}
This completes the proof.
\end{proof}

Using the locally constant property of the Weyl sums, first we estimate the number of the rectangles where the Weyl sums 
have the lower bound $N^{\alpha}$ for $\frac{d}{2} + \frac{d}{2(d+1)}<\alpha<d$.

\begin{prop}\label{app-1}
Let $\frac{d}{2} + \frac{d}{2(d+1)}<\alpha<d$, and $X$ be a union of rectangles $Q\subset[0,1]^{d+1}$ with dimension $(N^{-(d+1)+\alpha})^{d}\times N^{-2(d+1)+2\alpha}$, which is also
d-dimensional at scale $(N,\alpha)$. Let $C>0$ be a constant independent of other parameters. Suppose that for any $Q\in X$, there holds
\begin{equation}\label{lower-bound}
\|\omega_{N}(x,t)\|_{L^{\infty}(Q)} \ge CN^{\alpha}.
\end{equation}
Then we have
\begin{equation}
\#X\lesssim N^{(d^2+2d+2)(1-\frac{\alpha}{d})+\epsilon}.
\end{equation}
\end{prop}

\begin{proof}[\bf Proof] 
According to the assumption \eqref{lower-bound}, we have
\begin{align*}
&\Big(\sum_{Q\in X}\int_{Q}|\omega_{N}(x,t)|^{\frac{2(d+1)}{d}}\,\mathrm{d}x\mathrm{d}t\Big)^{\frac{d}{2(d+1)}}\\
\gtrsim& N^{\alpha}((N^{-(d+1)+\alpha})^{d}N^{-2(d+1)+2\alpha})^{\frac{d}{2(d+1)}}(\#X)^{\frac{d}{2(d+1)}}\\
=& N^{-(\frac{d^2}{2}+d)+\frac{(d^2+4d+2)\alpha}{2(d+1)}}(\#X)^{\frac{d}{2(d+1)}}.
\end{align*}
Besides, we  have
\begin{align*}
  &\Big(\sum_{Q\in X}\int_{Q}|\omega_{N}(x,t)|^{\frac{2(d+1)}{d}}\,\mathrm{d}x\mathrm{d}t\Big)^{\frac{d}{2(d+1)}}\\
  \le&  \Big(N^{-2(d+1)+2\alpha}\int_{\T^d}\sup_{0<t<1}|\omega_{N}(x,t)|^{\frac{2(d+1)}{d}}\,\mathrm{d}x\Big)^{\frac{d}{2(d+1)}}\\
  \le&  N^{\frac{d(2\alpha-d)}{2(d+1)}+\epsilon}.
\end{align*}
Combining these two estimates above, we have
\begin{equation*}
 N^{-(\frac{d^2}{2}+d)+\frac{(d^2+4d+2)\alpha}{2(d+1)}}(\#X)^{\frac{d}{2(d+1)}} \lesssim N^{\frac{d(2\alpha-d)}{2(d+1)}+\epsilon}.
\end{equation*}
This inequality yields
\begin{equation*}
\#X\lesssim N^{(d^2+2d+2)(1-\frac{\alpha}{d})+\epsilon}.
\end{equation*}
\end{proof}

For the proof of the maximal estimates \eqref{max-Lp}, recall that the maximal estimate
of the Weyl sums on the set 
$$\big\{x\in\T^d: \sup_{0<t<1}|\omega_{N}(x,t)|\le N^{\frac{d}{2}+\frac{d}{2(d+1)}}\big\}$$
is trival for the estimates \eqref{max-Lp}, and we just need to consider the contribution of the remaining
 $$\big\{x\in\T^d:\sup_{0<t<1}|\omega_{N}(x,t)|>N^{\frac{d}{2}+\frac{d}{2(d+1)}}\big\}.$$
It is necessary that the Lebesgue measure of
the set for $x\in\T^d$ where the maximal Weyl sums cannot be too large for the estimate \eqref{max-Lp}.
Employing the maximal estimates \eqref{max-Lp}
and Proposition \ref{app-1}, we can give a precise estimate for the Lebesgue measure of the set for $x\in\T^d$ where the maximal Weyl sums gain the large value as follows.

\begin{cor}\label{app-2}
For $\frac{d}{2}+\frac{d}{2(d+1)}<\alpha<d$, let
\begin{equation}
S_{\alpha}(N) = \big\{x\in\T^d: \sup_{0<t<1}|\omega_{N}(x,t)|\gtrsim N^{\alpha}\big\}.
\end{equation}
Then we have
\begin{equation}
  |S_{\alpha}(N)| \lesssim N^{(d+2)-\frac{2(d+1)\alpha}{d}+\epsilon}.
\end{equation}
\end{cor}

\begin{proof}[\bf Proof] 
By the locally constant property, for the Lebesgue measure of $S_{\alpha}(N)$,  we just need to consider the rectangles
$\prod_{j=1}^{d}[a_{j}N^{-(d+1)+\alpha},(a_{j}+1)N^{-(d+1)+\alpha}]$ for $0\le a_j\le N^{(d+1)-\alpha}$ with $1\le j\le d$.

By Proposition \ref{app-1}, we can obtain the number of these rectangles is at most $N^{(d^2+2d+2)(1-\frac{\alpha}{d})+\epsilon}$, and so
\begin{equation}
|S_{\alpha}(N)| \lesssim (N^{-(d+1)+\alpha})^{d}N^{(d^2+2d+2)(1-\frac{\alpha}{d})+\epsilon} = N^{(d+2)-\frac{2(d+1)\alpha}{d}+\epsilon}.
\end{equation}
\end{proof}

\begin{remark}
Corollary \ref{app-2} also follows from Chebychev's inequality and Theorem \ref{T-d}
with $p=\frac{2(d+1)}{d}$.
\end{remark}

\section{Hausdorff dimension of the large value set }

In this section, we consider lower and upper bounds of the Hausdorff dimension of the set $L_{\alpha}$ for
$\frac{d}{2}+\frac{d}{2(d+1)}\le\alpha\leq d$,
and give the proof of Proposition \ref{Hau-1}.
For the definition of $L_{\alpha}$, see \eqref{L-alpha} in Section 1.




\subsection{Upper bound}

To prove the upper bound in Proposition \ref{Hau-1}, we utilize a completion method of higher dimension
which is developed by Chen, Shparlinski in \cite{CS3}.
Meanwhile, we use a fractal maximal estimate of the
 periodic Schr\"odinger operator derived from Theorem \ref{T-d}, which can be obtained by the argument of Eceizabarrena and Luc\`a in \cite{EL}. Barron in \cite{A-B} studied the similar problem in one dimension.

Let
\begin{equation}
  L_{\alpha}(N) := \big\{x\in\T^d:\sup_{t\in \mathbb T}\big|\omega_{N}(x,t)\big|\geq N^\alpha\big\}.
\end{equation}
It is easy to see that
\begin{equation}\label{lim-sup}
  L_{\alpha}=\bigcap_{M\ge1}\bigcup_{N\ge M}L_{\alpha}(N).
\end{equation}
Fix $M\ge1$, then for $x\in L_{\alpha}$, define
\begin{equation}
  N(x) := \min\Big\{N\ge M:\sup_{t\in\T}|\omega_{N}(x,t)|\ge N^{\alpha}\Big\},
\end{equation}
and for $j\ge1$, define
\begin{equation}
  E_{j}:=  \{x\in L_{\alpha}:N(x)\in[2^{j-1}M, 2^{j}M)\}.
\end{equation}
From the expression \eqref{lim-sup}, we have
\begin{equation}
  L_{\alpha} \subset \bigcup_{j\ge1}E_j.
\end{equation}


We recall the lower bound of $S_{N,k}(x_k,t)$  given by Chen and Shparlinski in \cite{CS} and \cite{CS3}, where for $x_{k}\in\T$ with $1\le k\le d$, $S_{N,k}(x_k,t)$ is defined by
\begin{equation}\label{S-N-1}
  S_{N,k}(x_k,t):=\sum_{h=1}^{N}\frac{1}{h}|\sum_{n=1}^{N}e^{2\pi in\tfrac{h}{N}}e^{2\pi i(x_{k}n+tn^2)}|
  = \sum_{h=1}^{N}\frac{1}{h}|\omega_{N,k}(x_k+\tfrac{h}{N},t)|.
\end{equation}
\begin{lemma}\label{N<M}
  For all $(x_k,t)\in\T^2$ with $1\le k \le d$ and any $1\le N\le M$, we have
  \begin{equation}\label{S-M}
    |\omega_{N,k}(x_k,t)|\lesssim |S_{M,k}(x_k,t)|.
  \end{equation}
\end{lemma}
By this lemma, we can choose $M$ to be dyadic parameter to control the Weyl sums.
Indeed, for $x\in E_j$, by Lemma \ref{N<M} and the definition of $E_j$, we have
\begin{equation}\label{N<M-2}
  1\le N(x)^{-\alpha}\sup_{0<t<1}|\omega_{N(x)}(x,t)|
  \lesssim (2^{j}M)^{-\alpha}\sup_{0<t<1}|S_{2^{j}M}(x,t)|,
\end{equation}
where
\begin{equation}
\begin{split}
  S_{2^{j}M}(x,t)
  &: = \prod_{k=1}^{d}|S_{2^{j}M,k}(x_k,t)|\\
  &= \prod_{k=1}^{d}\Big|\sum_{h_k=1}^{2^{j}M}\frac{1}{h_k}\omega_{2^{j}M,k}(x_k+\tfrac{h_k}{2^{j}M},t)\Big|\\
  &= \Big|\sum_{\substack{1\le h_k\le 2^{j}M\\k=1,2,\cdots,d}}\prod_{k=1}^{d}\frac{1}{h_k}\omega_{2^{j}M,k}(x_k+\tfrac{h_k}{2^{j}M},t)\Big|\\
  &= \Big|\sum_{\substack{1\le h_k\le 2^{j}M\\k=1,2,\cdots,d}}\Big(\prod_{k=1}^{d}\frac{1}{h_k}\Big)\omega_{2^{j}M}(x+\tfrac{h}{2^{j}M},t)\Big|,
\end{split}
\end{equation}
and $S_{2^{j}M}(x_k,t)$ is defined by \eqref{S-N-1} and $h=(h_1, h_2, \cdots, h_d)$.

Recall that a Borel measure $\mu$ on $\T^d$ is said to be $\beta$-dimensional with $0<\beta\le d$ if
\begin{equation}
  c_{\beta}(\mu):=\sup_{B(x,r)}\frac{\mu(B(x,r))}{r^{\beta}} <\infty,
\end{equation}
where the supremum  is taken over balls $B(x,r)$ in $\T^d$.

\begin{lemma}[Frostman's Lemma, \cite{Mattila}]\label{Frost}
  Suppose $S\subset\T^d$ is a Borel set and $\mathcal{H}^{\gamma}(S)>0$ with $\gamma>0$.
  Then there exists a nonzero $\gamma$-dimensional measure $\mu_{\gamma}$ supported on $S$.
\end{lemma}
Frostman's Lemma establishes the relationship between the Hausdorff dimension of the Borel set
and corresponding Borel measure supported on it.

Next we cite the following useful lemma transforming Theorem \ref{T-d} into the fractal version, which plays an essential role in the  control of the measure for $\T^d$ with respect to the Borel measure supported on it.

\begin{lemma}[Proposition 5.2 in \cite{EL}]\label{mu-gamma} 
  Let $d\ge1$, $p\ge1$ and $s_0\ge0$. Fix $\epsilon>0$ and suppose $f\in H^{s_0+\epsilon}(\T^d)$ is a function such that
  \begin{equation}
    \Big\|\sup_{t\in\T}|e^{it\Delta_{\T^d}}f|\Big\|_{L^{p}(\T^d)}\lesssim \|f\|_{H^{s_0+\epsilon}(\T^d)}.
  \end{equation}
  Then for any $\gamma$-dimensional measure $\mu_{\gamma}$ on $\T^d$, we have
  \begin{equation}
    \Big\|\sup_{t\in\T}|e^{it\Delta_{\T^d}}f|\Big\|_{L^{p}(\T^d, \mathrm{d}\mu_{\gamma})}\lesssim \|f\|_{H^{s}(\T^d)},\quad s>\frac{d-\gamma}{p} + s_0,
  \end{equation}
\end{lemma}

Now we show how the upper bound of the Hausdorff dimension of the set $L_{\alpha}$ follows after combining Theorem \ref{T-d} with these three lemmas above.

For $\frac{d}{2}\le\alpha\le d$, we assume that $H^{\gamma}(L_{\alpha}) >0$ for
some $\gamma>0$. Then by
Lemma \ref{Frost}, there exists a $\gamma$-dimensional Borel measure $\mu_{\gamma}$
supported on $L_{\alpha}$.\\
{\bf Claim:} if $\gamma>\frac{2(d+1)}{d}(d-\alpha)$, then we have $\mu_{\gamma}(\T^d)=0$.



We first prove this claim, then show the equivalence between this claim
 and the upper bound for $\dim_{H} L_{\alpha}$ in Proposition \ref{Hau-1} in the end of this subsection.

By the estimates \eqref{N<M-2}, we have
\begin{equation}\label{S-j}
\begin{split}
  \mu_{\gamma}(\T^d)^{\frac{1}{p}} =\Big(\int_{\T^d}\,\mathrm{d}\mu_\gamma\Big)^{\frac{1}{p}}
  &\lesssim\Big(\sum_{j\ge1}(2^{j}M)^{-p\alpha}\int_{E_{j}}\sup_{0<t<1}|S_{2^{j}M}(x,t)|^p\,\mathrm{d}\mu_\gamma\Big)^{\frac{1}{p}}.
\end{split}
\end{equation}
For each $j\ge1$, by Lemma \ref{mu-gamma}, we obtain
\begin{equation}\label{S-j-2}
  \begin{split}
    &\Big(\int_{E_{j}}\sup_{0<t<1}|S_{2^{j}M}(x,t)|^p\,\mathrm{d}\mu_\gamma\Big)^{\frac{1}{p}}\\
   =&\Big(\int_{E_{j}}\sup_{0<t<1}\Big|\sum_{\substack{1\le h_k\le
     2^{j}M\\k=1,2,\cdots,d}}\Big(\prod_{k=1}^{d}\frac{1}{h_k}\Big)\omega_{2^{j}M}(x+\tfrac{h}{2^{j}M},t)\Big|^{p}\,\mathrm{d}\mu_\gamma\Big)^{\frac{1}{p}}\\
   \le&\sum_{\substack{1\le h_k\le2^{j}M\\k=1,2,\cdots,d}}\Big(\prod_{k=1}^{d}\frac{1}{h_k}\Big)
   \Big(\int_{E_{j}}\sup_{0<t<1}|\omega_{2^{j}M}(x+\tfrac{h}{2^{j}M},t)|^{p}\,\mathrm{d}\mu_\gamma\Big)^{\frac{1}{p}}\\
   =&\prod_{k=1}^{d}\Big(\sum_{\substack{1\le h_k\le2^{j}M\\k=1,2,\cdots,d}}\frac{1}{h_k}\Big)
   \Big(\int_{E_{j}}\sup_{0<t<1}|\omega_{2^{j}M}(x+\tfrac{h}{2^{j}M},t)|^{p}\,\mathrm{d}\mu_\gamma\Big)^{\frac{1}{p}}\\
   \lesssim_{\epsilon}&(2^{j}M)^{d\epsilon}(2^{j}M)^{\frac{d-\gamma}{p}+\frac{d}{2}+\frac{d}{2(d+1)}}.
  \end{split}
\end{equation}
Taking $p=\frac{2(d+1)}{d}$ and inserting \eqref{S-j-2} into \eqref{S-j}, we have
\begin{equation}\label{mu-gamma-2}
\begin{split}
  \mu_{\gamma}(\T^{d})^{\frac{d}{2(d+1)}}
  &\lesssim \Big\{\sum_{j\ge1}(2^{j}M)^{-\frac{2(d+1)}{d}\alpha}(2^{j}M)^{-\frac{2(d+1)}{d}d\epsilon}(2^{j}M)^{d-\gamma+d+2}\Big\}^{\frac{d}{2(d+1)}}\\
  &=M^{d-\alpha-\frac{d}{2(d+1)}\gamma-+d\epsilon}\Big(\sum_{j\ge1}2^{-j(\gamma+\frac{2(d+1)}{d}(\alpha-d)-2(d+1)\epsilon)}\Big)^{\frac{d}{2(d+1)}}.
\end{split}
\end{equation}
If $\gamma>\frac{2(d+1)}{d}(d-\alpha)$, say $\gamma=\frac{2(d+1)}{d}(d-\alpha)+\delta$. Then from \eqref{mu-gamma-2}, we have
\begin{equation}
  \mu_{\gamma}(\T^{d})^{\frac{d}{2(d+1)}}
  \lesssim M^{-\frac{d}{2(d+1)}\delta+d\epsilon}\Big(\sum_{j\ge1}2^{-j(\delta-2(d+1)\epsilon}\Big)^{\frac{d}{2(d+1)}}
  \lesssim M^{\frac{d}{4(d+1)}\delta},
\end{equation}
where we choose $\epsilon=\frac{d}{4(d+1)}\delta$. Let $M$ tend to $\infty$, we have $\mu_{\gamma}(\T^d)=0$, which proves the claim.


This claim and Lemma \ref{Frost} imply that $\mathcal{H}^{\gamma}(L_{\alpha})=0$ for $\gamma>\frac{2(d+1)}{d}(d-\alpha)$, which yields
\begin{equation}\label{upp}
\dim_{H}(L_{\alpha})\le \tfrac{2(d+1)}{d}(d-\alpha).
\end{equation}
Thus, the upper bound of the Hausdorff dimension for $L_{\alpha}$ in
Proposition \ref{Hau-1} has been proved.

\subsection{Lower bound}

Inspired by the results of Barron \cite{A-B}, one may conjecture  that the lower bound of  $\text{dim}_H L_\alpha$  is $\frac{2(d+1)(d-\alpha)}{d},$
if $\frac{d}{2}+\frac{d}{ 2(d+1) }<\alpha\leq d. $
Actually, this is indicated in the proof of Eceizabarrena and Luc\`{a} \cite{EL}.
Indeed, they show that there exists a function  $f\in H^s$ for $s<\frac{d}{2(d+1)}$ which consists of segmented Gauss sums fails to convergence on a nontrivial Hausdorff dimension set(See Theorem 1.1 of \cite{EL}).
Here, we explicate the proof of the lower bound of $\text{dim}_H L_\alpha$ via Proposition \ref{L^1-prop}.
%


By the analysis in Eceizabarrena and Luc\`{a} \cite{EL}, we have
\begin{lemma}[P. 16 in \cite{EL}]\label{EL}
Let  $\tau>\frac{d+1}d.$
Define
$$G_\tau=\{x\in \mathbb T^d: \big |qx_j-p_j\big|\leq q^{1-\tau}  \text{ holds for  infinitely many odd } q \text{  and  even } p_j  \}.$$
Then,  we have $\dim_{H} G_\tau = \frac{d+1}{\tau}.$
\end{lemma}

The ingredient of the proof for the estimate of the lower bound
for $\dim_{H} L_{\alpha}$ is the construction of  a subset included in $L_{\alpha}$
as in Lemma \ref{EL}, which is followed by the result immediately.

Take $q $ odd and $N_q$ such that
$q \sim (100)^{\frac{2(d-\alpha)}{d} } N_q^{\frac{2(d-\alpha)}{d} } <N_q^\frac{d}{d+1}. $
For $x$ such that
$$|x_j -\tfrac{b_j}{q} |<  q^{-\frac{d}{2(d-\alpha)}   } <  \tfrac{1}{100N_q}, \text{ \quad for some even}\,\, b_j \in \mathbb{Z}, j=1,\cdots,d.$$
By Proposition \ref{L^1-prop}, we have
\[
\sup_{t\in\mathbb T}|u(x,t)|\geq C \tfrac{N_q^{d}}{q^{\frac{d}{2}}} \sim N_q^{\alpha}.
\]
This yields $G_\frac{d}{2(d-\alpha)}   \subset L_\alpha$.
Thus, by Lemma \ref{EL}, we have
\begin{equation}\label{low}
\tfrac{2(d+1)(d-\alpha)}{d} =\dim_H G_\frac{d}{2(d-\alpha)}  \leq  \dim_H L_\alpha,
\end{equation}
which gives the proof for the lower bound of the Hausdorff dimension for $L_{\alpha}$ in
Proposition \ref{Hau-1} and completes the proof for Proposition \ref{Hau-1} together with
the estimate \eqref{upp}.

\begin{remark}
It is interesting to study the Fourier dimension of $G_\tau$ and it is reasonable to
conjecture $\dim_F G_\tau= \dim_H G_\tau = \frac{d+1}{\tau}$. If this conjecture holds, $G_\tau$ is  called a Salem set.


\end{remark}

\appendix

{\centerline{{\Large{\bf{Appendix}}}}

\section*{An alternate proof of Theorem \ref{T-d} inspired by Baker's argument}

\setcounter{equation}{0}
\renewcommand\theequation{A.\arabic{equation}}

\bigskip

\centerline{{\Large{by Alex Barron}}}

\bigskip

In this Appendix we give an alternate proof of Theorem \ref{T-d} which is more directly related to the argument employed by Baker in \cite{Baker}. Our argument relies on a straightforward higher-dimensional generalization of the Diophantine approximation result used by Baker in \cite{Baker}. Our proof slightly improves the $N^{\epsilon}$ loss in Theorem \ref{T-d} above. At the end of the Appendix we show that essentially the same argument also implies Strichartz-type estimates for the Weyl sum with logarithmic losses.

Our main tool is the following.

\begin{customprop}{A.1}\label{prop:multiBaker}
Let $w_N$ be the Weyl sum on $\T^{d+1}$ as in \eqref{omega-N}, given by $$w_{N}(x,t) = \sum_{ \substack{\textbf{n} \in \Z^d \cap [1,N]^{d} }} e^{2\pi i(x \cdot \textbf{n} + t|\textbf{n}|^2 )},$$ and suppose that $(x,t) \in \T^{d + 1}$ with $$ |w_{N}(x,t)| \geq N^{\alpha} \geq C_{d} N^{\frac{d}{2}}$$ for some sufficiently large constant $C_d$ that does not depend on $N$. Then if $N$ is sufficiently large there is an integer $q$ such that \begin{equation}\label{eq:tBoundHigherDim0} |t - a/q| \lesssim q^{-1}N^{-\frac{2}{d}\alpha} \ \ \text{ for some } \ a \leq q \ \text{ with } \  (a,q) = 1 \end{equation} and  \begin{equation}\label{eq:qBoundHigherDim0}  1 \leq q \lesssim  N^{2 - \frac{2}{d}\alpha},\end{equation} and moreover \begin{equation}\label{eq:higherDimBaker} \prod_{i=1}^{d}\|qx_i\| \lesssim N^{d -2\alpha}.\end{equation}
\end{customprop}
The proof of this proposition uses estimates that are standard if one ignores various losses on the order of $N^{\epsilon}$. Since we are interested in applying this theorem with no logarithmic losses we will sketch the requisite background before proving Proposition \ref{prop:multiBaker}.

We begin by defining $$I(\beta_1, \beta_2) = \int_{0}^{N} e^{2\pi i (\beta_1 \xi + \beta_2 \xi^2)} d\xi .$$ We have the following classical estimate, which is a consequence of integration by parts and the Van der Corput lemma for oscillatory integrals.

 \begin{customprop}{A.2}\label{prop:integralEst}
 One has $$|I(\beta_1, \beta_2)| \leq CN\max(1, N|\beta_1|, N^{2}|\beta_2|)^{-\frac{1}{2}}.$$
\end{customprop}
\noindent See also \cite{Vaughan1}, Chapter 7, for a more general version of this estimate. We will also use the following sharp decomposition of the one-variable quadratic Weyl sum due to Vaughan \cite{Vaughan2}.

\begin{customlemma}{A.3}[\cite{Vaughan2}, Theorem 8] \label{lem:bootstrap}
Suppose there are integers $1 \leq a,q \leq N$ with $(a,q) = 1$ such that $|t- a/q| \leq q^{-1}N^{-1}$. Also suppose $b$ is chosen so that $1 \leq b \leq q$ and $$|x - b/q| \leq 1/(2q).$$ Writing $$\beta_1 = x - b/q, \ \ \ \beta_2 = t - a/q,  $$ and $$ S(q) = \sum_{r = 0}^{q-1} e^{2\pi i ( \frac{b}{q}r + \frac{a}{q}r^2 )}, $$ we have \begin{equation} \label{eq:mainDecomp}\sum_{n=1}^{N} e^{2\pi i (xn + tn^2)} = q^{-1}S(q)I(\beta_1, \beta_2) + O(q^{1/2} + N|\beta_2|^{1/2}q^{1/2}). \end{equation}
\end{customlemma}
\noindent The key point here is the error term in \eqref{eq:mainDecomp}, which allows us to prove Proposition \ref{prop:multiBaker} with no logarithmic loss in $N$.

\begin{proof}[Proof of Proposition \ref{prop:multiBaker}]  We will use Lemma \ref{lem:bootstrap}, although we require some set-up. Let $w_{N,j}$ be a one-dimensional quadratic Weyl sum in the $j$-th variable. Choose $\gamma_j$ such that $$N^{\gamma_j} \leq |w_{N,j}(x_j, t)| \leq 2N^{\gamma_j}$$ for each $j$.
	
	Let $k$ be the number of $w_{N,j}$ such that $$|w_{N,j}(x_j, t)| \geq C_1 N^{\frac{1}{2}},$$ where $C_1$ is a sufficiently large constant to be chosen below. Note that we must have $k \geq 1$ if $C_d$ is sufficiently large since $|w_{N}(x,t)| \geq C_d N^{d/2}$. After relabeling we may assume that \begin{equation}\label{eq:unifLower} |w_{N,j}(x_j, t)| \geq N^{\gamma_j} \geq C_1  N^{\frac{1}{2}} \ \ \text{ for all } \ j = 1,...,k\end{equation} and \begin{equation} \label{eq:unifLower2} |w_{N,j}(x_j, t)| < C_1  N^{\frac{1}{2}} \ \ \text{ for } \ j = k+1, ..., d. \end{equation} Now by construction we have $$N^{\alpha} \leq \prod_{i=j}^{d}|w_{N,j}(x_j, t)| \leq  C N^{\gamma_1 + ... + \gamma_{k}}N^{\frac{1}{2}(d-k)}$$ and therefore \begin{equation} \label{eq:gammaBound} N^{\gamma_1 + ... + \gamma_k} \gtrsim N^{\alpha - \frac{1}{2}(d-k)}.\end{equation}
	
	We now apply Dirichlet's theorem to find $0 \leq a < q \leq N$ with $(a,q) = 1$ such that $|t - a/q| \leq q^{-1}N^{-1}$. We also find $1 \leq b_j \leq q$ such that $|x_j - b_j /q| \leq 1/(2q)$. For $j = 1, ..., k$ we have $|w_{N,j}(x_j, t)| \geq C_1 N^{1/2}$. We now apply the decomposition in Lemma \ref{lem:bootstrap}. If $C_1$ is sufficiently large the error term is negligible, and therefore the classical Gauss sum bound $|S(q)|\lesssim q^{1/2}$ yields \begin{equation} \label{eq:iBound} |w_{N,j} (x_j, t)| \lesssim  q^{-1/2} N \max(1, N|x_j - b_j/q|, N^2 |t -a/q| )^{-1/2}, \ \ j = 1,..., k. \end{equation} With \eqref{eq:unifLower2} this implies $$N^{\alpha} \leq |w_{N}(x,t)| \lesssim q^{-k/2}|t-a/q|^{-k/2}N^{\frac{d-k}{2}},$$ and since $|t - a/q| \leq q^{-1}N^{-1}$ we have $N^{1/2} \leq q^{-1/2}|t-a/q|^{-1/2}$ and hence \begin{equation}\label{eq:tBoundMulti} N^{\alpha} \lesssim q^{-d/2} |t-a/q|^{-d/2}. \end{equation} The estimate \eqref{eq:tBoundHigherDim0} now follows after rearranging \eqref{eq:tBoundMulti}. From \eqref{eq:iBound} we also have $$q \lesssim N^{2 - 2\gamma_j}, \ \ j = 1,..., k,$$ and therefore by \eqref{eq:gammaBound} $$q \lesssim N^{2 - \frac{2}{k}(\gamma_1 + ... + \gamma_k)} \lesssim N^{2- (\frac{2}{k}\alpha - \frac{d-k}{k})}.$$ Now $$ \frac{2}{k} \alpha - \frac{d-k}{k} \geq \frac{2}{d}\alpha $$ since $\alpha \geq d/2$, so \eqref{eq:qBoundHigherDim0} follows. Finally, to prove \eqref{eq:higherDimBaker} we note that \eqref{eq:iBound} implies \begin{equation} \label{eq:bakerEst1} \|qx_j \| \lesssim N^{1-2\gamma_j} \ \ \text{ for } \ \ j = 1,...,k, \end{equation} Then from \eqref{eq:bakerEst1} and \eqref{eq:gammaBound} we obtain
	\begin{align*}\prod_{j=1}^{k} \|qx_j \| &\lesssim  N^{k - 2(\gamma_1 + ... + \gamma_k)} \lesssim N^{d - 2\alpha  }. \end{align*} But $\|qx_j\| < 1$ for each $j$, so \eqref{eq:higherDimBaker} follows.
\end{proof}

We now show that Proposition \ref{prop:multiBaker} implies Theorem \ref{T-d} with $N^{\epsilon}$ loss replaced by $\log(N)^{\frac{d^2}{2(d+1)}}$. To begin the proof note that by pigeonholing it suffices to show that if $$E^{d}_{\alpha} = \{x \in \T^d : N^{\alpha} \leq \sup_{0<t<1}|w_{N}(x,t)| < 2N^{\alpha} \} $$ for some dyadic $N^{\alpha}$ then $$\big( \int_{E_{\alpha}} \sup_{0 < t < 1}|w_{N}(x,t)|^\frac{2(d+1)}{d} dx \big)^{\frac{d}{2(d+1)}} \lesssim \log(N)^{(d-1)\frac{d}{2(d+1)}} N^{\frac{d}{2(d+1)} + \frac{d}{2}}.$$ Moreover, we can assume that $$\frac{d}{2(d+1)} + \frac{d}{2} < \alpha \leq d$$ since otherwise the estimate is trivial. In particular Proposition \ref{prop:multiBaker} applies.

It suffices to prove the following.
\begin{customprop}{A.4} \label{prop:measureEst}
Suppose $\alpha > \frac{d}{2(d+1)} + \frac{d}{2}.$ Then  \begin{equation}\label{eq:measureEst} |E^{d}_{\alpha}| \lesssim \log(N)^{d-1} N^{d+2 -\frac{2(d+1)}{d}\alpha}. \end{equation}
\end{customprop}

\begin{proof}  Our proof will rely on the Diophantine approximation properties of the points where $w_{N}$ can attain large values, as quantified by Proposition \ref{prop:multiBaker}. Let us momentarily fix $x \in E^{d}_{\alpha}$. By Proposition \ref{prop:multiBaker} there must be an integer $1 \leq  q \lesssim N^{2 - \frac{2}{d}\alpha}$ and integers $1 \leq b_j \leq q$ such that \begin{equation} \label{eq:prodEst} \prod_{i=1}^{d}|x_j - b_j/q| \lesssim q^{-d}N^{d - 2\alpha}.\end{equation} Given a vector $\textbf{b} = (b_1,...,b_{d}) \in \Z^d$ with $1 \leq b_j \leq q,$ let $\mathcal{A}(q;\textbf{b}) \subset \T^d$ denote the set of $x$ for which \eqref{eq:prodEst} holds.
	\begin{customlemma}{A.5}\label{lem:volEst}
We have $|\mathcal{A}(q;\textbf{b})| \lesssim \log(N)^{d-1} q^{-d}N^{d-2\alpha}.$
	\end{customlemma}
	\begin{proof} Note that $\mathcal{A}(q;\textbf{b})$ is contained in a union of $O(\log(N)^{d-1})$ dyadic rectangles, each of volume proportional to $q^{-d}N^{d-2\alpha}$. Alternatively one can write $|\mathcal{A}(q;\textbf{b})|$ as a $d$-dimensional integral and the result then follows from an elementary calculation. \end{proof}
	
	Now by the above discussion we must have $$E_{\alpha}^{d} \subset \bigcup_{q=1}^{cN^{2- \frac{2}{d}\alpha}} \bigcup_{\substack{\textbf{b} \in \Z^d \\ 1 \leq b_j \leq q }} \mathcal{A}(q;\textbf{b}),$$ and therefore by Lemma \ref{lem:volEst} \begin{align*}
	|E_{\alpha}^{d}| &\lesssim \log(N)^{d-1}N^{d-2\alpha }\sum_{q=1}^{c N^{2- \frac{2}{d}\alpha}} \sum_{\substack{\textbf{b} \in \Z^d \\ 1 \leq b_j \leq q }}q^{-d} \\ & \lesssim \log(N)^{d-1} N^{d - 2\alpha} N^{2- \frac{2}{d}\alpha} \\&  \lesssim \log(N)^{d-1} N^{d + 2 -\frac{2(d+1)}{d}\alpha }.
	\end{align*} This proves Proposition \ref{prop:measureEst}.  \end{proof}

The desired estimate now follows immediately from Proposition \ref{prop:measureEst}, as discussed before. 

We end the Appendix by remarking that Proposition \ref{prop:multiBaker} also implies the following mean-value or Strichartz-type estimate.

\begin{customprop}{A.6}\label{prop:strichartz}
Let $p_{d} = \frac{2(d+2)}{d}$. We have $$\|w_{N}\|_{L^{p_d} (\T^{d+1})} \leq C \log(N)^{\sigma} N^{\frac{d}{2}},$$ with $\sigma = \frac{d^2}{2(d+2)} $
\end{customprop}
\noindent This was proved by Hu and Li in \cite{Hu-Li} for Weyl sums with a loss in $N$ that is slightly worse than logarithmic (resulting from the divisor bound). It also follows with constant $N^{\frac{d}{2} + \epsilon}$ from the much more general Strichartz estimates proved by Bourgain and Demeter in \cite{BD} using decoupling. It was shown by Bourgain in \cite{Bourgain-1} that one needs $\sigma \geq \frac{d}{2(d+2)}.$

\begin{proof}The proof is similar to the proof of Proposition \ref{prop:measureEst}. By pigeonholing we see it suffices to show that $$\int_{S_{\alpha}}|w_{N}(x,t)|^{p_d} \  dxdt \lesssim \log(N)^{d-1}N^{d + 2},$$ where $$F_{\alpha}= \{(x,t) \in \T^{d} \times \T : N^{\alpha} \leq |w_{N}(x,t)| < 2N^{\alpha} \}.$$ Using Proposition \ref{prop:multiBaker} as above, we see that $$F_{\alpha} \subset \bigcup_{q=1}^{cN^{2- \frac{2}{d}\alpha}} \bigg( \bigcup_{\substack{\textbf{b} \in \Z^d \\ 1 \leq b_j \leq q }} \mathcal{A}(q;\textbf{b}) \times \bigcup_{ \substack{1 \leq a \leq q } }\{t: |t-a/q| \leq Cq^{-1}N^{-\frac{2}{d}\alpha} \} \bigg).$$ Therefore \begin{align*}|F_{\alpha}| &\lesssim \log(N)^{d-1}\sum_{q=1}^{cN^{2- \frac{2}{d}\alpha}} \sum_{\substack{\textbf{b} \in \Z^d \\ 1 \leq b_j \leq q }} \sum_{1 \leq a \leq q}q^{-(d+1)} N^{d-2\alpha } N^{-\frac{2}{d}\alpha} \\ & \lesssim \log(N)^{d-1} N^{d + 2 -\frac{2(d+1)}{d}\alpha } N^{-\frac{2}{d}\alpha} \\ &\lesssim \log(N)^{d-1} N^{d+2} N^{-\frac{2(d+2)}{d}\alpha} \end{align*} as desired.
	
\end{proof}

{\bf Acknowledgements.}
The authors are very grateful to Alex Barron for adding an appendix, which gives another new proof
for Theorem \ref{T-d} and is very precious and helpful for the readers to understand.
The authors would also like to thank the associated editor and anonymous referee for their helpful comments and suggestions which helped improve the paper greatly.
C. Miao was supported by the National Key Research
and Development Program of China (No. 2020YFA0712900) and NSFC Grant 11831004.
T. Zhao was supported by the NSFC Grant 12101040 and the Fundamental Research Funds for the Central Universities (FRF-TP-20-076A1).

\end{document}